\documentclass[article]{IEEEtran}

\usepackage{cite}

\usepackage{amsmath,amssymb,amsfonts,amsthm}
\usepackage{xcolor}
\usepackage{textcomp}
\usepackage{mathrsfs}
\usepackage{enumitem}
\usepackage{caption}
\usepackage{algorithm}
\usepackage{bbm}
\usepackage{algpseudocode}
\usepackage{tikz}
\usepackage{graphics}
\usepackage{epsfig}
\usepackage{verbatim}
\usepackage{float}
\usepackage{graphicx}
\usepackage{multirow}
\usepackage{epstopdf}
\usepackage{color}

\pdfoutput=1

\usetikzlibrary{shapes.geometric, arrows, positioning, arrows.meta, decorations.markings}

\newtheorem{definition}{Definition}[section]%
\newtheorem{theorem}[definition]{Theorem}%
\newtheorem{proposition}[definition]{Proposition}%
\newtheorem{lemma}[definition]{Lemma}%
\newtheorem{assumption}[definition]{Assumption}%
\newtheorem{corollary}[definition]{Corollary}%
\newtheorem{remark}[definition]{Remark}%
\newtheorem{example}[definition]{Example}%

\usepackage{color}
\newcommand{\sy}[1]{{\color{magenta} #1}}

\begin{document}

\title{Approximations and Learning for Decentralized Stochastic Control and Near Optimal Finite Window Policies}

\author{Omar Mrani-Zentar and Serdar Y\"{u}ksel%
\thanks{O. Mrani-Zentar and S. Y\"uksel are with the Department of Mathematics and Statistics, Queen's University, Kingston, Ontario, Canada.}}

\maketitle

\begin{abstract}
Decentralized stochastic control problems are difficult to study due to information structure dependent subtleties, which prevent many classical methods in stochastic control from being applicable. In this paper we consider such problems with general standard Borel spaces under two related information structures. (a) the one-step delayed information sharing pattern (OSDISP) where agents share their information with one-step delay, and (b) the $K$-step periodic information sharing pattern (KSPISP), where information is shared periodically. It is known that OSDISP and KSPISP problems admit a centralized reduction where the agents view the problem from the perspective of a centralized controller that uses the common information to prescribe function valued actions (local policies) which map each agent's private information to an optimal action in the original problem. We provide rigorous approximation results and performance bounds for the KSPISP and OSDISP problems, which results from replacing the full common information by a finite sliding window of information and we establish near optimality of such policies. The latter depends on a predictor stability condition in expected total variation. As a further contribution, we show that under the information structures provided, corresponding Q-learning algorithms (in quantized or finite memory forms) converge asymptotically to near optimal solutions. While restrictive and hypothetical conditions have been presented in the literature, our contributions are thus to provide, to our knowledge, the first explicit conditions and rigorous approximation and learning results for such decentralized problems with general spaces.
\end{abstract}

\begin{keywords}
  Decentralized Stochastic Control, Partially Observed Markov Decision Processes, Information Structures
\end{keywords}

\section{Introduction and Problem Description} \label{problem description}

An increasingly important research area of mathematical and practical interest is {\it decentralized stochastic control} or {\it team theory}, which involves multiple decision makers (DMs) who strive for a common goal but who have access only to some noisy measurements regarding the state of their environment. Applications include energy systems, the smart grid, sensor networks, and networked control systems, among others \cite{wit75} (see \cite{YukselBasarBook24} for a comprehensive review). There still exist problems (such as Witsenhausen's counterexample \cite{wit68}) which have defied solution attempts for more than 50 years. Existence and rigorous approximation results have only recently been obtained. Accordingly, learning theory for such systems entails many open problems. 

In this paper, we consider such a decentralized stochastic control model with $N \in \mathbb{N}$ agents with general state, measurement, and action spaces and under a variety of information structures. 

Let $\mathbb{X};\mathbb{U}^{1},\cdots,\mathbb{U}^N;\mathbb{Y}^{1},\cdots,\mathbb{Y}^N$ be locally compact subsets of standard Borel spaces, i.e., Borel subsets of metric Polish spaces (complete, and separable metric spaces); here $\mathbb{X}$ denotes the state space, the action spaces for agents $1,...,N$ are $\mathbb{U}^{1},...,\mathbb{U}^N$, and their measurement spaces are given by $\mathbb{Y}^{1},...,\mathbb{Y}^N$. We denote by $\mathcal{B}(\mathbb{X})$ the Borel $\sigma$-algebra on $\mathbb{X}$ and the same notation holds for all spaces. At every time $t \in \mathbb{Z}_+$ the system dynamics evolve as follows:
\begin{eqnarray} 
   \label{Prob 1} x_{t+1}&=&\Phi(x_{t},u^{1}_{t},...,u^{N}_{t},v_{t})    \\
   \label{obs 1} \forall i \in \{1,..,N\} \text{: } y^{i}_{t}&=&\phi^i(x_{t},w^{i}_{t}) 
\end{eqnarray}
where $\{ v_{t}\}_{t}$ and $\{ w^{i}_{t}\}_{t}$ are i.i.d noise processes for all $i \in \{1,..,N\}$. Additionally, $\Phi$ and $\phi^i$ are Borel measurable functions.  We assume that all random variables are defined on a common probability space $(\Omega, {\cal F}, P)$. The dynamics and measurement equations above induce the following kernels for any Agent $i$: let $A \in {\cal B}(\mathbb{X})$ and $B \in {\cal B}(\mathbb{Y}^{i})$: 
\begin{align}
  &  \mathcal{T}(A|x,u^{1},...,u^{N})\nonumber \\
   &  \qquad \qquad :=P(x_{t+1} \in A|x_t=x,u^1_t=u^{1},...,u^N_t=u^{N})\\
  &  Q^{i}(B|x):=P(y_{t} \in B|x_t=x)
\end{align}

At every time $t$, each Agent $i$ has access to their local information 
\begin{align}\label{localInf}
I_{t}^{i}=(y_{[0,t]}^{i}, u_{[0,t-1]}^{i})
\end{align}
as well as some common information $I_{t}^{C}$, to be discussed below. We denote the spaces in which the random variables $I_{t}^{i}$ and $I_{t}^{C}$ take values by $\mathbb{I}_{t}^{i}$ and $\mathbb{I}_{t}^{C}$ respectively. Each Agent applies a policy $\gamma^{i}=\{\gamma^{i}_{t} \}_{t\in \mathbb{Z}^{+}} \in \Gamma^i$, such that $\gamma_{t}^{i}:\mathbb{I}_{t}^{C}\times \mathbb{I}_{t}^{i}\rightarrow \mathbb{U}^{i}$ and is $\sigma(I_{t}^{i},I_{t}^{C})$-measurable, with the objective of minimizing 
\begin{equation} \label{cost}
   J(\gamma)= E^{\gamma}[\sum_{t=0}^{\infty} \beta^{t} c(x_{t},\mathbf{u_{t}})], 
\end{equation}
where $\beta \in (0,1)$ is the discount factor, $\gamma=(\gamma^{1},...,\gamma^{N})$ is the team policy, and $c: \mathbb{X} \times \mathbb{U} \rightarrow \mathbb{R}$ is a cost functional which is assumed to be continuous and bounded. We refer to the set of all admissible team policies as $\Gamma:= \prod\limits_{k=1}^N \Gamma^{k}$. We also assume that $x_{0}\sim \mu$ and that the initial distribution $\mu$ is known to all agents at time $t=0$. Next, we will describe the information structures that will be considered in this paper. 

\subsection{Information structures} \label{information structures}
We define the various information structures that we will consider in this paper. In all cases, each Agent has access to their own measurements and previous actions given in (\ref{localInf}). However, the additional information shared among agents, i.e., the common information will vary.
\begin{itemize}
    \item[(i)] \textit{One-step delayed information sharing pattern (OSDISP)}: \[I_{t}^{C}=\{y_{[0,t-1]}^{[1,N]},u_{[0,t-1]}^{[1,N]}\}.\]
    \item[(ii)] $K$\textit{-step periodic information sharing pattern (KSPISP)}: Let $t=qK+r$ where $q$ and $r$ are non-negative integers such that $r\leq K-1$ and $K\in \mathbb{N}$ represents the period between successive information sharing between the Agents. Here the common information at time $t$ is given by $I^{C}_{t}=\{y^{[1,N]}_{[0,qK-1]},u^{[1,N]}_{[0,qK-1]}\}$.
    \item[(iii)] \textit{Completely decentralized information structure (CDIS)}: $I_{t}^{C}=\{\emptyset\}$ 
\end{itemize}
\begin{remark}
    We note that when $K=1$, the KSPISP problem simply reduces to the OSDISP problem. However, to establish many of the useful properties for the KSPISP problem it will be beneficial to first establish those properties for the OSDISP problem, which also present more relaxed conditions. We also note that while this paper does not address the CDIS directly, an important motivation for the KSPISP is that a solution to the KSPISP as $K$ becomes large provides improved lower bounds for CDIS.
\end{remark}

\subsection{Literature Review and Contributions}
 
 Decentralized stochastic control problems are also inherently difficult because of their non-classical information structures \cite{Witsenhausen}, \cite[Section 3.2]{YukselBasarBook24} which render tools designed for centralized control problems inapplicable. 
 
 Nonetheless, dynamic programming methods are known to be applicable by lifting the problem. One way is via what is now commonly termed as the {\it common information approach} \cite{NayyarMahajanTeneketzis,NayyarMahajanTeneketzis2,NayyarBookChapter}, which also builds on prior studies such as \cite{yos75,VaraiyaWalrand,OoiWornell,YukTAC09,kur79,NayyarMahajanTeneketzis2}. While the above focused on finite, linear, or otherwise more restrictive spaces, \cite{Saldi2023Commoninformationapproachstaticteamproblems} considered static problems involving standard Borel spaces.
  
 For the OSDISP, it was shown in \cite{VaraiyaWalrand,yos75} that optimal policies are of separated form when the state, action, and measurement spaces are all finite. In \cite{OMZSY-centralized-reduction} this was extended to the case of standard Borel spaces where it was shown that OSDISP admits a centralized reduction whereby the state consists of the one-step predictor conditioned on the common information, and the action at any given time consists of a vector of functions where each function maps the private information of some agent to their action. Additionally, the centralized reduction is weak-Feller which, via measurable selection conditions \cite[Chapter 3]{HernandezLermaMCP}, enables one to apply dynamic programming \cite{OMZSY-centralized-reduction}. In \cite{OMZSY-centralized-reduction} this was used to establish existence results for OSDISP problems and will be crucial for establishing the convergence of a quantized Q-learning algorithm we propose to attain a near-optimal solution. For the KSPISP, it was shown in \cite{OoiWornell} (see also \cite{YukTAC09}) that optimal policies are of separated form when the state space, action spaces, and measurement spaces are all finite. Under these same assumptions it was shown in \cite{YukTAC09} that such problems admit a centralized reduction. This was later generalized in \cite{OMZSY-centralized-reduction} for the case of general standard Borel spaces. As will be seen in this paper, the weak-Feller regularity will be crucial for the development of numerical learning techniques with performance guarantees. 

 Another dynamic programming approach is via considering unconditional probability measures and their evolution under measure valued actions; see \cite{WitsenStandard,YukselWitsenStandardArXiv}.

The numerical complexity of decentralized stochastic control problems has been a significant impediment to the development of suitable numerical methods \cite{bernstein2002complexity}. As a result, the approach taken in the literature to numerical learning for decentralized problems consists mainly of algorithms that converge but with no theoretical performance guarantees  (see \cite{mao2023decentralized,gupta2017cooperative,hu2019simplified,wu2013monte,liu2015stick,kraemer2016multi,Zhou-pomdp-comp} for example). 
 
When considering a decentralized stochastic control problem under an infinite horizon discounted cost criterion, an issue arises due to fact that depending on the information structure either the private or common information is growing over time which prevents numerical learning and leads to prohibitive memory constraints. Thus, compressing the growing information becomes of critical importance. In \cite{tavafoghi2018unified}, it is shown that that when a certain compression function exists for the private information, optimal policies which rely on such a compression are near optimal, however, it is not shown how to obtain such a compression function. Similarly, this result is extended in \cite{Subramanian-2021-common} where the compression of both common and private information is studied based on compression functions which satisfy certain conditions. However, again, it is not shown how such compression functions could be obtained. \cite{LiuZhang2025} also consider such implicit compression functions for the common information but under different assumptions. As noted in \cite{LiuZhang2025}, it is not clear in general how such compression functions could be obtained. 

For centralized POMDPs the near optimality of finite memory policies was established in \cite{kara2021CDCconvergence}\cite{kara2020near,kara2021convergence} (for related and somewhat more restrictive setting please see \cite{golowich2022planning}) and fore more recent studies \cite{demirciRefined2023,cayci2024finite}. For studies which allow for general information structures see \cite{LiuYouZhang2025,altabaa2024role}. In \cite{altabaa2024role}, a more general framework is presented in terms of partially observable sequential teams. In \cite{LiuYouZhang2025}, the information shared by the agents appear as a design parameter which the agents can optimize based on the common information at their disposal at any given time. 

\subsection{Statements of Main Results and Examples}\label{Main_cont_section} 

 For KSPISP and OSDISP problems two primary challenges arise due to the increasing common information and a need for Bayesian computation updates. While implicit conditions for the compression of common information are available notably in \cite{Subramanian-2021-common,LiuZhang2025}, building also on the approximate model analysis under uniform error bounds \cite{subramanian2022approximate}, the imposed conditions are typically too demanding (see Definition 1 \cite{Subramanian-2021-common} and Definition 7 \cite{LiuZhang2025}) and not explicit given a system: As noted by \cite{LiuZhang2025} ''it is in general unclear how to construct such an $\mathcal{M}$" where $\mathcal{M}$ refers to approximate models satisfying the implicit conditions. 
 
 We present (i) a more relaxed framework (where restrictive uniformity conditions in approximation error bounds are not imposed), (ii) consider general (possibly uncountable) spaces, and (iii) more importantly, we present explicit conditions by presenting and building on a weak Feller continuity analysis and explicit filter/predictor stability conditions. These allow for a more general and practical setting (see Figure \ref{information structures figure} for a summary of the results involving approximate optimality). In particular, in this paper, we address this by establishing near optimality of finite length sliding window of common information under a predictor stability condition which we also establish under easily computable conditions as will be demonstrated later through examples. In \cite{SaLiYuSpringer}, it was shown how one can approximate a general weak Feller Markov decision process (MDP) with a \textit{quantized} MDP which has finite state and action spaces (whose solutions are near optimal for the original model), and in \cite{KSYContQLearning} the rigorous convergence of a Q-learning algorithm for problems with standard Borel spaces to such a near optimal solution was established. Building on and generalizing these studies, we will present two Q-learning algorithms and establish their convergence to near optimal solutions. We then provide a further approximation result involving finite window of information. These, to our knowledge, are the first learning algorithms with explicit conditions and performance guarantees developed for KSPISP and OSDISP problems with general spaces. A major benefit of our analysis is that a near optimal policy can be obtained by interacting with the original environment from which the problem emanates; this contrasts with the approach previously taken in the literature which rests on the existence of oracles that provide simulations of an implicit and approximate model \cite{LiuYouZhang2025}.    

\begin{itemize}
 
\item[(i)] \textbf{Approximations} Exact sharing of information in KSPISP (and the equivalent belief-sharing information sharing pattern \cite{OoiWornell,YukTAC09}) is impractical for many applications. To this end, we provide a rigorous finite memory approximation method and result for any general $K$-step periodic belief sharing pattern using a sliding window of information, so that the agents only share a finite window of measurements and actions instead of their beliefs, and establish bounds for the loss in optimality due to such approximations via a predictor stability condition (see Theorems \ref{boundonperformance} and \ref{PredStabHilbert}). We then show that under technical conditions on the transition and measurements kernels, the predictor is exponentially stable under expected total variation (see Theorem \ref{predictor stability one-step delayed theorem}) and under expected Hilbert metric (see Theorem \ref{PredStabHilbert}). We thus obtain a rigorous approximation which is near optimal. Figure \ref{information structures figure} summarizes the systematic program leading to finite memory approximations. 

\item[(ii)] \textbf{Learning} We present an asynchronous Q-leaning algorithm (AQPQ), see Algorithm \ref{alg:AQPQ} and a synchronous Q-learning algorithm (SQPQ), see Algorithm \ref{alg:SQPQ}. We then establish their almost sure convergence to a near optimal solution under the $K$-step periodic information sharing pattern (see Theorem \ref{Convergence of Quantized Q-learning}). Algorithms  \ref{alg:AQPQ},\ref{alg:SQPQ} are based on sharing the predictor amongst agents and quantizing the latter. This requires computationally demanding Bayesian updates and require prior knowledge of the model. This motivates a third algorithm with sliding window periodic sharing pattern (QSWPSP), see Algorithm \ref{alg:QSWPSP}, which is a reinforcement learning algorithm that builds on our rigorous approximation results for the KSPISP and also converges to a near optimal solution. We, then, apply our Q-learning algorithms to two examples for which rigorous conditions for convergence to a near optimal solution are established (see Section \ref{K step numerics}  and Section \ref{sliding mem numerics}).  
\end{itemize}

\allowdisplaybreaks
\begin{figure}[p] \label{information structures figure}
\centering

\begin{minipage}{0.5\textwidth}
\centering
\text{(a)}
\vspace{0.2cm}

\scalebox{0.99}{%
\begin{tikzpicture}[>=Stealth, node distance=0.6cm]
\node[draw, minimum width=1.4cm, minimum height=1cm] (f1) {$\gamma_t^{1}$};
\draw[->] ([xshift=-2.5cm]f1.west) -- (f1.west)
    node[midway, above] {$I_t^{1}=(I^{C}_{t},I^{P^{1}}_{t})$};
\node[right=1.5cm of f1] (u1) {$u_t^{1}$};
\draw[->] (f1.east) -- (u1);

\node[draw, minimum width=1.4cm, minimum height=1cm, below=0.1cm of f1] (f2) {$\gamma_t^{2}$};
\draw[->] ([xshift=-2.5cm]f2.west) -- (f2.west)
    node[midway, above] {$I_t^{2}=(I^{C}_{t},I^{P^{2}}_{t})$};
\node[right=1.5cm of f2] (u2) {$u_t^{2}$};
\draw[->] (f2.east) -- (u2);

\node[below=0.01cm of f2] (dots) {$\vdots$};

\node[draw, minimum width=1.4cm, minimum height=1cm, below=0.5cm of dots] (fN) {$\gamma_t^{N}$};
\draw[->] ([xshift=-2.5cm]fN.west) -- (fN.west)
    node[midway, above] {$I_t^{N}=(I^{C}_{t},I^{P^{N}}_{t})$};
\node[right=1.5cm of fN] (uN) {$u_t^{N}$};
\draw[->] (fN.east) -- (uN);
\end{tikzpicture}%
}

\vspace{0.5cm}
\text{(c)}

\scalebox{0.7}{%
\begin{tikzpicture}[>=Stealth, node distance=0.6cm]

\node[align=center] at (-5,-2) (piq) {$\pi_{q-M}$,
\\ $\mathbf{y}_{[(q-M)K,qK]}$,
\\ $\mathbf{u}_{[(q-M)K,qK-1]}$};

\coordinate (fcol) at (3,0);

\node[draw, align=center, minimum width=1.4cm, minimum height=1cm,
      below=0.01cm of fcol] (f1b)
      {$f_{qK+r}^{1}:I^{P^{1}}_{[qK,qK+r]}\rightarrow u_{qK+r}^{1}$\\
       $r\in \{0,\cdots, K-1\}$};
\draw[->] (piq.east) -- (f1b.west);

\node[draw, align=center, minimum width=1.4cm, minimum height=1cm,
      below=1cm of f1b] (f2b)
      {$f_{qK+r}^{2}:I^{P^{2}}_{[qK,qK+r]}\rightarrow u_{qK+r}^{2}$\\
       $r\in \{0,\cdots, K-1\}$};
\draw[->] (piq.east) -- (f2b.west);

\node[below=0.1cm of f2b] (dots2) {$\vdots$};

\node[draw, align=center, minimum width=1.4cm, minimum height=2cm,
      below=1cm of f2b] (fNb)
      {$f_{qK+r}^{N}:I^{P^{N}}_{[qK,qK+r]}\rightarrow u_{qK+r}^{N}$\\
       $r\in \{0,\cdots, K-1\}$};
\draw[->] (piq.east) -- (fNb.west);

\end{tikzpicture}%
}

\end{minipage}
\hfill
\begin{minipage}{0.48\textwidth}
\centering

\text{(b)}
\vspace{0.2cm}

\scalebox{0.8}{%
\begin{tikzpicture}[>=Stealth, node distance=0.6cm]

\node at (-4,-2) (piq) {$\pi_{q}$};

\coordinate (fcol) at (0,0);

\node[draw, align=center, minimum width=1.4cm, minimum height=1cm,
      below=0.01cm of fcol] (f1b)
      {$f_{qK+r}^{1}:I^{P^{1}}_{[qK,qK+r]}\rightarrow u_{qK+r}^{1}$\\
        $r\in \{0,\cdots, K-1\}$};
\draw[->] (piq.east) -- (f1b.west);

\node[draw, align=center, minimum width=1.4cm, minimum height=1cm,
      below=1cm of f1b] (f2b)
      {$f_{qK+r}^{2}:I^{P^{2}}_{[qK,qK+r]}\rightarrow u_{qK+r}^{2}$\\
        $r\in \{0,\cdots, K-1\}$};
\draw[->] (piq.east) -- (f2b.west);

\node[below=0.1cm of f2b] (dots2) {$\vdots$};

\node[draw, align=center, minimum width=1.4cm, minimum height=1cm,
      below=1cm of f2b] (fNb)
      {$f_{qK+r}^{N}:I^{P^{N}}_{[qK,qK+r]}\rightarrow u_{qK+r}^{N}$\\
        $r\in \{0,\cdots, K-1\}$};
\draw[->] (piq.east) -- (fNb.west);

\end{tikzpicture}%
}

\vspace{0.2cm}
\text{(d)}
\vspace{0.2cm}

\scalebox{0.8}{%
\begin{tikzpicture}[>=Stealth, node distance=0.6cm]

\node[align=center] at (-1,-2) (piq)
  {$\hat{\pi}$,
  \\ $\mathbf{y}_{[(q-M)K,qK]}$,
  \\ $\mathbf{u}_{[(q-M)K,qK-1]}$};

\coordinate (fcol) at (3,0);

\node[draw, align=center, minimum width=1.4cm, minimum height=1cm,
      below=0.01cm of fcol] (f1b)
      {$f_{qK+r}^{1}:I^{P^{1}}_{[qK,qK+r]}\rightarrow u_{qK+r}^{1}$\\
        $r\in \{0,\cdots, K-1\}$};
\draw[->] (piq.east) -- (f1b.west);

\node[draw, align=center, minimum width=1.4cm, minimum height=1cm,
      below=1cm of f1b] (f2b)
      {$f_{qK+r}^{2}:I^{P^{2}}_{[qK,qK+r]}\rightarrow u_{qK+r}^{2}$\\
       $r\in \{0,\cdots, K-1\}$};
\draw[->] (piq.east) -- (f2b.west);

\node[below=0.1cm of f2b] (dots2) {$\vdots$};

\node[draw, align=center, minimum width=1.4cm, minimum height=1cm,
      below=1cm of f2b] (fNb)
      {$f_{qK+r}^{N}:I^{P^{N}}_{[qK,qK+r]}\rightarrow u_{qK+r}^{N}$\\
       $r\in \{0,\cdots, K-1\}$};
\draw[->] (piq.east) -- (fNb.west);

\end{tikzpicture}%
}

\end{minipage}

\caption{Equivalent formulations and approximation of KSPISP: Figure (a) corresponds to the original problem description where $I^{P^{i}}_{t}=\{y^{i}_{[0,t],u^{i}_{[0,t]}}\}$. Figure (b) shows the centralized reduction of KSPISP with $t=qK+r$ and $\pi_{q}(.)=P(x_{qK}\in.|I^{C}_{qK})\in \mathcal{P}(\mathbb{X})$. Figure (c) shows an equivalent formulation using the sliding window periodic sharing pattern (SWPSP), and Figure (d) shows a finite memory approximation of the latter towards near optimal design where $\hat{\pi}$ denotes some arbitrary but fixed probability measure on $\mathbb{X}$ used to replace $P(dx_{(q-M)K}\in.|I^{C}_{(q-M)K})$.}
\label{information structures figure}
\end{figure}
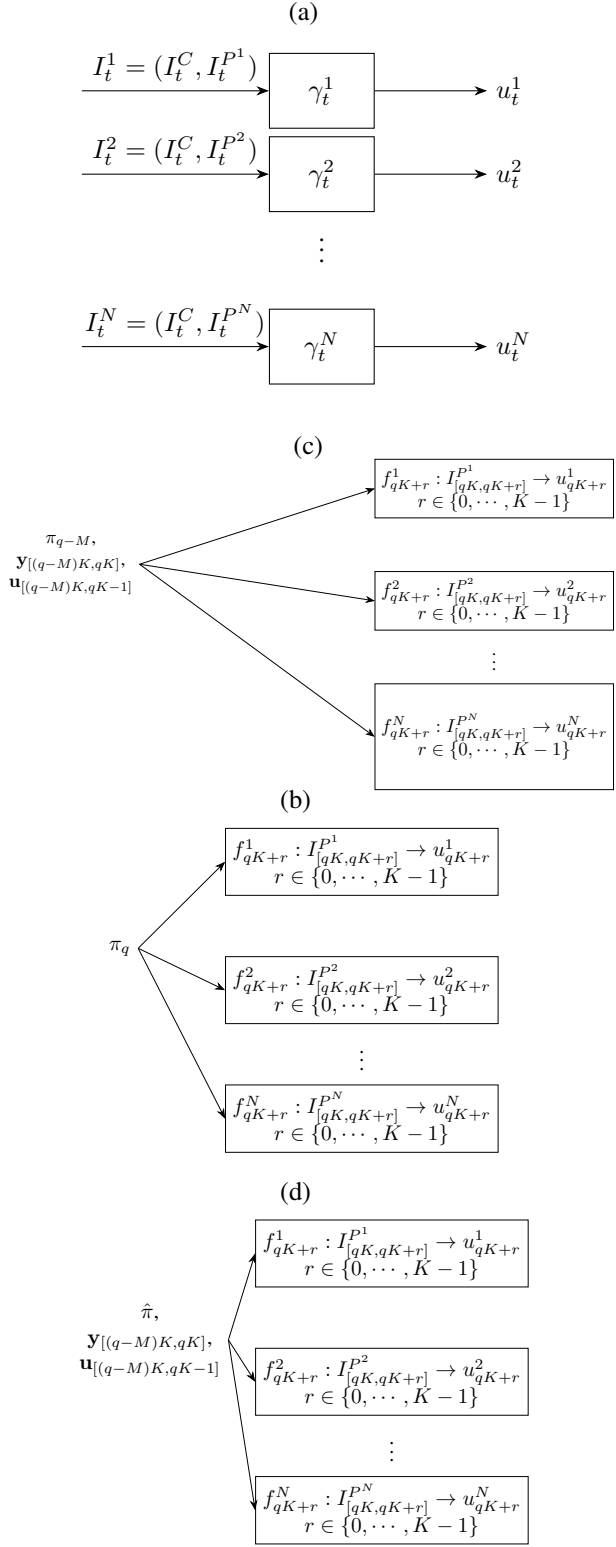

\subsection{Notation and Preliminaries}

\subsubsection{Notation}

   
   (i) For any standard Borel set $\mathbb{X}$ we denote the set of probabilities on $\mathbb{X}$ by $\mathcal{P}(\mathbb{X})$. (ii) We will denote a finite sequence $\{x_{t}\}_{t=l}^{t=k}$, where $k,l\in \mathbb{Z}^{+}$ by $x_{[l,k]}$ or $x^{[l,k]}$. (iii)  Similarly, for a double-indexed sequence $\{x_{t}^{i}\}_{\{t\in\{l,...,k\},i\in\{l',...,k'\}\}}$ we will use the notation $x^{[l',k']}_{[l,k]}$. (iv) Whenever the superscript of a double-indexed sequence ranges over all agents ($i\in \{1,...,N\}$) and when there is no ambiguity, we will use the notation $\mathbf{x}_{[l,k]}$ to refer to the double-indexed sequence $\{x_{t}^{i}\}_{\{t\in\{l,...,k\},i\in\{1,...,N\}\}}$. (v) Moreover, we will use the notation $\mathbb{U}=\mathbb{U}^{1}\times...\times\mathbb{U}^{N}$ and $\mathbb{Y}=\mathbb{Y}^{1}\times...\times\mathbb{Y}^{N}$. (vi) We will use $C_{b}(\mathbb{X})$ to denote the set of continuous and bounded functions on the space $\mathbb{X}$. (vii)  For a measurable function $f:(\mathbb{X},\mathcal{F}_{X})\rightarrow (\mathbb{U},\mathcal{B}(\mathbb{U}))$ where $\mathcal{F}_{X}$ is a $\sigma$-algebra on $\mathbb{X}$ and $\mathbb{U}$ is standard Borel, we will use the notation 
   \[ f(du|x): x \mapsto f(du|x) \in \mathcal{P}(\mathbb{U}),\]
   to refer to the conditional probability induced by $f$ on $\mathbb{U}$. In particular, for any $x\in \mathbb{X}$ $f(.|x)=\delta_{f(x)}(.)$ where $\delta_{f(x)}(.)$ is the Dirac measure concentrated at $f(x)$.

      Throughout this paper all probability valued spaces are endowed with the weak topology which we introduce below. Moreover, all product spaces are endowed with the product topology.

\subsubsection{Convergence of probability measures} \label{convergence of probability measures}
In this section, we will introduce three notions of convergence for sequences of probability measures that will be useful later on: weak convergence, convergence under total variation, and convergence under the Wasserstein metric of order one. For a complete, separable and metric space $\mathbb{X}$, a sequence $\left\{\mu_n \right\}_{n \in \mathbb{N}}$ $\subset$ $\mathcal{P}(\mathbb{X})$ is said to converge to $\mu \in \mathcal{P}(\mathbb{X})$ weakly if and only if $\int_{\mathbb{X}} f(x) \mu_n(d x) \rightarrow \int_{\mathbb{X}} f(x) \mu(d x)$ for every continuous and bounded $f: \mathbb{X} \rightarrow \mathbb{R}$. One important property of weak convergence is that the space of probability measures on a complete, separable, metric (Polish) space endowed with the topology of weak convergence is itself complete, separable, and metric. An example of such a metric on $\mathcal{P}(\mathbb{X})$ is the bounded Lipschitz metric, which is defined for $\mu, \nu \in \mathcal{P}(\mathbb{X})$ as
$$
\rho_{B L}(\mu, \nu):=\sup _{\|f \|_{BL\leq 1}}\left|\int f d \mu-\int f d \nu\right|
$$
where
$$
\|f\|_{B L}:=\|f\|_{\infty}+\sup _{x \neq y} \frac{|f(x)-f(y)|}{d(x, y)}
$$
and $\displaystyle \|f\|_{\infty}=\sup _{x \in \mathbb{X}}|f(x)|$. Here, $d(.,.)$ denotes the metric on $\mathbb{X}$. 

For probability measures $\mu, \nu \in \mathcal{P}(\mathbb{X})$, the total variation metric is given by
\begin{eqnarray*}
    &&\|\mu-\nu\|_{T V}=2 \sup _{B \in B(\mathbb{X})}|\mu(B)-\nu(B)|\\
    &&=\sup _{\|f\|_{\infty \leq 1}}\left|\int f(x) \mu(\mathrm{d} x)-\int f(x) \nu(\mathrm{d} x)\right|,
\end{eqnarray*}
\noindent where the supremum is taken over all measurable real-valued functions $f$ such that $\displaystyle \|f\|_{\infty}=\sup_{x \in \mathbb{X}}|f(x)| \leq$ 1. A sequence $\mu_{\mathrm{n}}$ is said to converge in total variation to $\mu \in \mathcal{P}(\mathbb{X})$ iff $\| \mu_n-\mu\|_{T V} \rightarrow 0$.
The Wasserstein metric of order one is given by
\begin{equation*}
    W_{1}(\mu,\nu)=\sup_{\|f\|_{\text{Lip}}\leq 1}\bigg| \int f d\mu - \int f d\nu   \bigg| 
\end{equation*}
 where $\|f\|_{\text{Lip}}:=\sup _{x \neq y} \frac{|f(x)-f(y)|}{d(x, y)}$

We now provide several examples that demonstrate the continuity of a transition kernel under various metrics. This will be useful later on as the continuity of the measurement channels $Q^{i}$ in $x$, and the continuity of $\mathcal{T}$ in $x,u$ will be needed for the continuity of the transition kernel that results from a centralized reduction \cite{OMZSY-centralized-reduction}. We note that when the space $\mathbb{X}$ is finite, continuity of  of $\mathcal{T}$ in $x,u$ under one of the metrics mentioned above implies continuity in the other as they are all equivalent in the sense that they generate the same topology. For centralized Markov decision processes various examples leading to different regularity conditions are presented in \cite{KSYContQLearning}. Recall (\ref{Prob 1}).

\begin{example} Suppose the realization function $\Phi(x,u^{1},\cdots,u^{N},v)$ is continuous in $(x,u^{1},\cdots,u^{N})$ for every $v$. Then, the kernel $\mathcal{T}$ is weak-Feller (weakly continuous), i.e., for any sequence $\{(x_{n},u^{1}_{n},\cdots,u^{N}_{n})\}_{n\in \mathbb{N}}$ such that $(x_{n},u^{1}_{n},\cdots,u^{N}_{n})\rightarrow (x,u^{1},\cdots,u^{N})$ as $n\rightarrow \infty$, we have that for any continuous and bounded function $g\in \mathbf{C_{b}}(\mathbb{X})$ $\int g(x_{1})\mathcal{T}(dx_{1}|x_{n},u^{1}_{n},\cdots,u^{N}_{n})\rightarrow \int g(x_{1})\mathcal{T}(dx_{1}|x,u^{1},\cdots,u^{N})$ as $n\rightarrow \infty$. To see this, let $g\in \mathbf{C_{b}}(\mathbb{X})$. Then, by an application of the dominated convergence theorem (DCT), we have that
\begin{eqnarray*}
    &&\int g(x_{1})\mathcal{T}(dx_{1}|x_{n},u^{1}_{n},\cdots,u^{N}_{n})\\
    && \quad =\int g(\Phi(x_{n},u^{1}_{n},\cdots,u^{N}_{n},v))P(dv)\\
    &&\rightarrow\int g(\Phi(x,u^{1},\cdots,u^{N},v))P(dv)\\
    &&\quad =\int g(x_{1})\mathcal{T}(dx_{1}|x,u^{1},\cdots,u^{N})
\end{eqnarray*}
Thus, this model satisfies the continuity condition on the kernel as needed for Theorem \ref{theorem5}. 
\end{example}

\begin{example} Suppose the kernel $\mathcal{T}$ admits a density $\tilde{\Phi}$ with respect to some measure $\xi\in \mathcal{P}(\mathbb{X})$ so that $\mathcal{T}(x_{1}\in\cdot|x,u^{1},\cdots,u^{N})=\int_{\cdot}\tilde{\Phi}(x_{1},x,u^{1},\cdots,u^{N})\xi(dx_{1})$. If $\tilde{\Phi}$ is continuous in $(x,u^{1},\cdots,u^{N})$, then $\mathcal{T}$ is continuous in total variation, since for any sequence $\{(x_{n},u^{1}_{n},\cdots,u^{N}_{n})\}_{n\in \mathbb{N}}$ such that $(x_{n},u^{1}_{n},\cdots,u^{N}_{n})\rightarrow (x,u^{1},\cdots,u^{N})$ as $n\rightarrow \infty$, we have that 
\begin{eqnarray*}
     &&\sup_{\|g\|_{\infty}\leq 1}\bigg|\int g(x_{1})\mathcal{T}(dx_{1}|x_{n},u^{1}_{n},\cdots,u^{N}_{n})\\
     && \quad \quad \quad \quad -\int g(x_{1})\mathcal{T}(dx_{1}|x,u^{1},\cdots,u^{N})\bigg|\rightarrow 0
\end{eqnarray*}
as $n\rightarrow \infty$. This follows from an application of Scheff\'e Lemma \cite[Theorem 16.12]{Bil95}. 
This then serves as an example for the conditions on the kernels for both Theorems \ref{theorem5} and \ref{theorem6}.
\end{example}

\begin{example} Suppose the realization function $\Phi$ can be written as $\Phi(x,u^{1},\cdots,u^{N},v)=\tilde{\Phi}(x,u^{1},\cdots,u^{N})+v$ where $\tilde{\Phi}$ is continuous in $(x,u^{1},\cdots,u^{N})$ and $v$ admits a continuous density function $\iota$ with respect to some measure $\xi\in \mathcal{P}([0,1])$. Then, $\mathcal{T}$ is continuous in total variation. This then serves as an example for the conditions on the kernels for both Theorems \ref{theorem5} and \ref{theorem6}. Similarly, one can construct examples for which regularity of the observation channels hold.
\end{example}

\section{Supporting Results: Centralized Reductions, Weak-Feller Regularity, and Existence} \label{supporting-results}

In this section we review and introduce several key supporting results.
\subsection{One-step delayed information sharing pattern}
Let the predictor at time $t$ be given by $Z_{t}(\cdot)=P(x_{t} \in \cdot |y_{[0,t-1]}^{[1,N]},u_{[0,t-1]}^{[1,N]})$ and
$f^{i}_{t}:y^{i}_{t} \mapsto u_{t}^{i},$ be $\mathcal{B}(Y^{i})$-measurable and represent the {\it action} of Agent $i$ at time $t$. Let $\Bar{\Gamma}^{i}_{t}$ denote the space where $f^{i}_{t}$ takes values.

 \begin{theorem} [Theorem 4.5 \cite{OMZSY-centralized-reduction}] \label{them2young}
  Suppose the team policy at time $t$ is given by $(\Tilde{\gamma}_{t}^{1},...,\Tilde{\gamma}_{t}^{N})$ such that for all $i$, $\Tilde{\gamma}^{i}_{t}:\mathcal{P}(\mathbb{X})\ni Z_{t} \mapsto f^{i}_{t}\in \Bar{\Gamma}^{i}_{t}$ is $\mathcal{B}(\mathcal{P}(\mathbb{X}))$-measurable. Then, $( Z_{t} ,f_{t})$, where $f_{t}=(f^{1}_{t},...,f^{N}_{t})$,  forms a fully observed MDP and is equivalent to the one-step delayed information sharing pattern problem. 
  \end{theorem}

\noindent{\bf Young topology on Maps as Actions.} Here, we will introduce Young topology on local control policies (see, e.g. \cite{BorkarRealization,yuksel2023borkar}) which map private information to actions. That is to say, we will introduce a topology on $\Bar{\Gamma}^{i}_{t}$. Suppose there exists a reference measure $\psi$ and a measurable function $\displaystyle h^{i}: \mathbb{X} \times  \mathbb{Y}^{i}\rightarrow [0,\infty)$ such that the measurement channel, of each individual agent $i$, $Q^{i}$ satisfies:
\begin{equation} \label{Young top}
    Q^{i}(dy^{i}_{t}|x_{t})=\int \psi(dy^{i}_{t})h^{i}(x_{t},y^{i}_{t})
\end{equation}
For the sake of readability, throughout the rest of the paper, we will assume that $h^{i}\equiv h$ for all $i \in \{1,...,N\} $ where $h$ is a measurable function such that $\displaystyle h: \mathbb{X} \times  \mathbb{Y}^{i}\rightarrow \mathbb{R}$. As noted in \cite{OMZSY-centralized-reduction} the more general case proceeds in a similar manner.
\begin{definition} \label{topology for actions onestep} For all $i \in \{1,...,N\}$ let $\{f^{i}_{n,t}\}$ be a sequence of measurable functions such that $f^{i}_{n,t}:\mathbb{Y}^{i} \rightarrow \mathbb{U}^{i}$ and let $f^{i}_{t}:\mathbb{Y}^{i} \rightarrow \mathbb{U}^{i}$ be a measurable function. We say
$f_{n,t}=(f^{1}_{n,t},...,f^{N}_{n,t})\rightarrow f^{i}_{t}=(f^{1}_{t},...,f^{N}_{t})$ under the Young topology if and only if for all $i \in \{1,...,N\}$ and for all $g \in C_{b}(\mathbb{Y}^{i},\mathbb{U}^{i})$ : 
\begin{eqnarray*}
  \int g(y^{i},u^{i})f^{i}_{n,t}(du^{i}|y^{i})\psi(dy^{i})\rightarrow \int g(y^{i},u^{i})f^{i}_{t}(du^{i}|y^{i})\psi(dy^{i}).  
\end{eqnarray*}

\end{definition}
\begin{remark}
    It follows that, since the marginal on the measurement space is fixed, the convergence in Definition \ref{topology for actions onestep} is equivalent to the requirement that
    $\int g(y^{i},u^{i})f^{i}_{n,t}(du^{i}|y^{i})\psi(dy^{i})\rightarrow \int g(y^{i},u^{i})f^{i}_{t}(du^{i}|y^{i})\psi(dy^{i})$ for any measurable function $g$ which is bounded and continuous in $y^{i}$ (\cite[Theorem 3.10]{Schal},  \cite[Theorem 2.5]{balder2001}).
\end{remark}

This topology leads to a convex and compact formulation (see Section 2 in \cite{BorkarRealization} as well as Section 3 in \cite{Bor91}). In particular, consider the set of probability measures on $\mathbb{Y}^{i} \times \mathbb{U}^{i}$ with fixed marginal $\psi\in \mathcal{P}(\mathbb{Y}^{i})$: 
\begin{eqnarray*}
    E^{i}&=&\big\{P\in \mathcal{P}(\mathbb{Y}^{i} \times \mathbb{U}^{i}) \mid P(B)=\int_{B} f(du^{i}|y^{i})\psi(dy^{i}),\\
    &&\text{ for some } f\in \mathcal{F}^{i} \text{ and for all } B\in\mathcal{B}(\mathbb{Y}^{i} \times \mathbb{U}^{i})  \big \}
\end{eqnarray*}
where $\mathcal{F}^{i}$ is the space of regular conditional probability measures from $\mathbb{Y}^{i}$ to $\mathbb{U}^{i}$ (see \cite[section 10.2]{Dud02} or \cite[Appendix 4]{dynkin1979controlled1}): 
\begin{eqnarray*}
    \mathcal{F}^{i}=\big\{f:(\mathbb{Y}^{i},\mathcal{B}(\mathbb{Y}^{i})) \rightarrow (\mathcal{P}(\mathbb{U}^{i}),\mathcal{B}(\mathcal{P}(\mathbb{U}^{i})))| \text{ }f \text{ is measurable}  \big \}
\end{eqnarray*}
Recall that $\mathcal{P}(\mathbb{U}^{i})$ is endowed with the weak topology. Then, every deterministic map $f^{i}: \mathbb{Y}^{i} \rightarrow \mathbb{U}^{i}$ can be identified as an element of the set of extreme points of $E^{i}$ \cite[Lemma 2.2]{BorkarRealization}:
\begin{eqnarray*}
   \nonumber \Theta^{i}&=&\big\{P\in \mathcal{P}(\mathbb{Y}^{i} \times \mathbb{U}^{i}) \mid P(B)=\int \mathbbm{1}_{\{(y^{i},g(y^{i}))\in B\}}\psi(dy^{i}),\\
   && \text{ for \, some \, measurable \,} \text{function } g:\mathbb{Y}^{i}\rightarrow \mathbb{U}^{i}\\
   &&\text{ and for all } B\in\mathcal{B}(\mathbb{Y}^{i} \times \mathbb{U}^{i})  \big \}
\end{eqnarray*}
Thus, $\Theta^{i}$ inherits the Borel measurability and topological properties
of the Borel measurable set $E^{i}$ \cite[Section 2]{BorkarRealization}. 

For the MDP  $( Z_{t} ,f_{t})$ the kernel is given by 
\begin{eqnarray*}
    &&\eta(Z_{t+1} \in \cdot | Z_{t}, f_{t}^{[1,N]}):=P(Z_{t+1} \in \cdot |Z_{t},\mathbf{f}_{t})\\
&&=\int \mathbbm{1}_{ \{ F(Z_t,\mathbf{u}_{t},y_{t}^{[1,N]}) \in \cdot \}} Z_{t}(dx_{t})\prod\limits_{i=1}^{N} Q^{i}(dy_{t}^{i}|x_{t})f^{i}_{t}(du^{i}_{t}|y^{i}_{t})
\end{eqnarray*}
where, 
\begin{eqnarray} \label{equation for F}
      && F(Z_{t},\mathbf{u}_{t},y_{t}^{[1,N]})(\cdot):=Z_{t+1}(\cdot)\\
       \nonumber &&= \int \int_{\cdot}  \frac{\mathcal{T}(dx_{t+1}|x_{t}, u_{t}^{[1,N]})Z_{t}(dx_{t})\prod\limits_{i=1}^{N} h(x_{t},y_{t}^{i})}{\int_{\mathbb{X}}\prod\limits_{i=1}^{N} h(x_{t},y_{t}^{i})Z_{t}(dx_{t})}
\end{eqnarray}
The cost is $
    J(\tilde{\gamma})=\sum_{t=0}^{\infty} \beta^{t} E^{\tilde{\gamma}}[\tilde{c}({Z}_{t},f^{1}_{t},...,f^{N}_{t})]$,  
where \begin{eqnarray}
    &&\tilde{c}({Z}_{t},f^{1}_{t},...,f^{N}_{t})=\\
    \nonumber &&\int Z_{t} (dx_{t})\prod\limits_{i=1}^{N}f^{i}_{t}(du^{i}_{t}|y^{i}_{t})Q^{i}(dy^{i}_{t}|x_{t})c(x_{t},u^{1}_{t},...,u^{N}_{t}).
\end{eqnarray} 
We next present a theorem on the weak continuity of the kernel $\eta$ that will be crucial for establishing the convergence of the quantized Q-learning algorithm which will be introduced in Section \ref{LearningQSectionInf}.
\begin{theorem} [Theorem 5.1 \cite{OMZSY-centralized-reduction}]\label{theorem5}
Suppose the transition kernel $\mathcal{T}(x_{t+1}\in \cdot |x_{t}=x,\mathbf{u}_{t}=\mathbf{u})$ is weak-Feller, i.e., weakly continuous in $x$ and $u$ that is to say for any $(x_{n},\mathbf{u}_{n})\rightarrow (x,\mathbf{u})$ and any $g\in \mathbf{C_{b}}(\mathbb{X})$: $\int g(x)\mathcal{T}(dx|x_{n},\mathbf{u}_{n})\rightarrow \int g(x)\mathcal{T}(dx|x,\mathbf{u})$. Additionally suppose $h(x,y)$ is continuous in $x$ (for all $y$: $h(.,y): \mathbb{X} \rightarrow \mathbb{R}$ is a continuous function) and suppose that $\mathbb{X}$ is compact. Then, the MDP $(Z_{t},f_{t})$ introduced in Theorem \ref{them2young} is weak-Feller.
\end{theorem}
Theorem \ref{theorem5} implies the following existence result.
\begin{corollary} \label{existence OSDISP}
    Suppose the transition kernel $\mathcal{T}(x_{t+1}\in \cdot |x_{t}=x,\mathbf{u}_{t}=\mathbf{u})$ is weakly continuous in $x$ and $\mathbf{u}$. Additionally suppose $h(x,y)$ is continuous in $x$ and that $\mathbb{X}$ is compact. Then, an optimal solution to the problem (and equivalently to the centralized reduction) of the OSDISP problem (\ref{Prob 1}),(\ref{obs 1}),(\ref{cost}) exists and stationary policies are optimal.
\end{corollary}

\subsection{$K$-step periodic information sharing pattern} \label{K periodic section}

For problem (\ref{Prob 1},\ref{obs 1}), let $t=qK+r$ where $q$ and $r$ are non-negative integers such that $r\leq K-1$ and $K\in \mathbb{N}$ represent the period between successive information sharing between the Agents. Define
\begin{align} \label{k-step info}
   \nonumber  I^{i}_{t}&=\{y^{i}_{[qK,qK+r]},u^{i}_{[qK,qK+r-1]},I^{C}_{t}\} \\
     I^{C}_{t}&=\{y^{[1,N]}_{[0,qK-1]},u^{[1,N]}_{[0,qK-1]}\} \\
   \nonumber  I^{P^{i}}_{t}&= \{y^{i}_{[qK,qK+r]},u^{i}_{[qK,qK+r-1]}\} 
\end{align} 
Suppose that each Agent $i$ has access to $I^{i}_{t}$; in the above, $I^{P^{i}}$ denotes the private information and $I^{C}$ denotes the common information.
 
Here, each Agent $i$ at time $t$ picks a policy $\gamma^{i}_{t}: I^{i}_{t} \mapsto u^{i}_{t}$ such that $\gamma^{i}_{t}$ is $\mathcal{B}(\mathbb{I}^{i}_{t} )$-measurable. The objective of the agents is to minimize 
\begin{equation} \label{Cost k-step}
   J(\gamma)=\sum_{t=K}^{\infty} \beta^{t} E^{\gamma}[c(x_{t},\mathbf{u_{t}})] 
\end{equation}

The next proposition shows that one can discretize the action spaces.

\begin{proposition}[Proposition 3.1 \cite{OMZSY-centralized-reduction}] \label{finite action approximation}
    Suppose the state space $\mathbb{X}$ and the action space $\mathbb{U}=\prod\limits_{i=1}^{N}\mathbb{U}^{i}$ are compact. Additionally, assume that the transition kernel $\mathcal{T}$ is continuous, under total variation, in $x$ and $u$ that is to say for any $(x_{n},\mathbf{u}_{n})\rightarrow (x,\mathbf{u})$: 
    \[\|\mathcal{T}(x.|x_{n},\mathbf{u}_{n})-\mathcal{T}(x.|x,\mathbf{u})\|_{TV}\rightarrow 0.\] Let $\mathbb{U}_{m}=\{\mathbf{u}_{1},...,\mathbf{u}_{N_{m}}\}$ be such that for any $\mathbf{u}\in \mathbb{U}$ there exists $\mathbf{u}_{j}\in \mathbb{U}_{m}$ such that $d(\mathbf{u},\mathbf{u}_{j})\leq \frac{1}{m}$ where $m\in \mathbb{N}$. Let $\Gamma_{m}$ denote the set of admissible team policies with range in $\mathbb{U}_{m}$. Then, for any $\epsilon>0$ there exists $m'$ such that for all $m\geq m'$ we have that $\displaystyle |\inf_{\gamma\in \Gamma}J(\gamma)-\inf_{\gamma\in \Gamma_{m}}J(\gamma)|\leq \epsilon$.
\end{proposition}
Hence, throughout the rest of the paper we will assume that for all $i$: $\mathbb{U}^{i}$ is finite. This allows for the introduction of the following topology on the space of policies. 

\noindent{\bf Topology on the actions for the K-step periodic information sharing pattern}
Here, we introduce a topology on the space of local policies that map the private information to actions. Again, suppose there exists a reference measure $\psi\in \mathcal{P}(\mathbb{Y}^{i})$ and a measurable function $\displaystyle h: \mathbb{X} \times \bigcup_{i=1}^{N} \mathbb{Y}^{i}\rightarrow [0,\infty)$ such that the measurement channel, of each individual Agent $i$, $Q^{i}$ satisfies: \begin{equation} \label{Young top}
    Q^{i}(dy^{i}_{t}|x_{t})=\int \psi(dy^{i}_{t})h(x_{t},y^{i}_{t})
\end{equation}

\begin{definition} \label{top k-step}
For $i\in \{1,...,N\}$ and for all $n\in \mathbb{N}$ let $f^{i}_{qK,n}:\mathbb{Y}^{i}\rightarrow\mathbb{U}^{i}$ be a measurable function and let $f^{i}_{qK}:\mathbb{Y}^{i}\rightarrow\mathbb{U}^{i}$ be a measurable function. We say $f^{i}_{qK,n}\rightarrow f^{i}_{qK}$ if and only if $\mathcal{P}(\mathbb{Y}^{i}\times \mathbb{U}^{i})\ni \mu_{n}(.)=\int_{\cdot}\psi(dy^{i})f^{i}_{qK,n}(du^{i}_{qK}|y^{i}_{qK})\rightarrow \mu(.)=\int_{\cdot }\psi(dy^{i}_{qK})f^{i}_{qK}(du^{i}_{qK}|y^{i}_{qK})$ weakly, i.e., for all $g\in\mathbf{C_{b}}(\mathbb{Y}^{i}\times \mathbb{U}^{i})$:
\begin{eqnarray*}
&&\int g(y^{i},u^{i})\psi(dy^{i})f^{i}_{qK,n}(du^{i}_{qK}|y^{i}_{qK}) \\
&& \quad \quad \rightarrow \int g(y^{i},u^{i})\psi(dy^{i})f^{i}_{qK}(du^{i}_{qK}|y^{i}_{qK})
\end{eqnarray*}

Additionally, for $r>0$ we say $f^{i}_{qK+r,n}\rightarrow f^{i}_{qK+r}$ if and only if for all $u^{i}_{qK},...,u^{i}_{qK+r-1}$:   $\displaystyle \mathcal{P}((\mathbb{Y}^{i})^{r}\times \mathbb{U}^{i})\ni \mu_{n}(.)=\int_{\cdot}\prod\limits_{t=qK}^{t=qK+r}\psi(dy^{i}_{t})f^{i}_{qK+r,n}(du^{i}_{qK+r}|y^{i}_{[qK,qK+r]},u^{i}_{[qK,qK+r-1]})\rightarrow \mu \in \mathcal{P}((\mathbb{Y}^{i})^{r}\times \mathbb{U}^{i})$ weakly. Where, 
$\mu(.) =\int_{\cdot}\prod\limits_{t=qK}^{t=qK+r}\psi(dy^{i}_{t})f^{i}_{qK+r}(du^{i}_{qK+r}|y^{i}_{[qK,qK+r]},u^{i}_{[qK,qK+r-1]})$. Define for each $i\in\{1,...,N\}$: 
$a^{i}_{q}=(f^{i}_{qK},...,f^{i}_{(q+1)K-1})$. Then, we say that $a^{i}_{q,n}=(f^{i}_{qK,n},...,f^{i}_{(q+1)K-1,n})\rightarrow a^{i}_{q}=(f^{i}_{qK},...,f^{i}_{(q+1)K-1})$ if and only if for all $r\in \{0,...,K-1\}$: $f^{i}_{qK+r,n}\rightarrow f^{i}_{qK+r}$.
\end{definition}
For $t=qK+r$,  let $\pi_{t}(\cdot):=\pi_{q}(\cdot):=Z_{qK}=P(x_{qK} \in \cdot|y^{[1,N]}_{[0,qK-1]},u^{[1,N]}_{[0,qK-1]})$ and for each $i$, let $a^{i}_{q}=(f^{i}_{qK},...,f^{i}_{(q+1)K-1})$, where $f^{i}_{qK+r}:I^{P^{i}}_{qK+r}\mapsto u^{i}_{t}$. We use the notation $\tilde{\Gamma}^{i}_{q}$ to denote the space where $a^{i}_{q}$ takes values and $\mathbb{A}$ to denote the space where $(a^{1}_{q},...,a^{N}_{q})$ takes values. 

Next, we present an import result on the centralized reduction of KSPISP problems. 
    \begin{theorem} [Corollary 4.15 \cite{OMZSY-centralized-reduction}] \label{centralized reduction KSPISP}
        $ (\pi_{q},(a^{1}_{q},...,a^{N}_{q}); q\geq 1)$ forms a controlled Markov decision process. Moreover, $(\pi_{q},(a^{1}_{q},...,a^{N}_{q}); q\geq 1)$ is equivalent to the KSPISP problem under an appropriate choice of the cost criteria, and optimal policies are of the form $\gamma^{i}_{q}: \mathcal{P}(\mathbb{X})\ni\pi_{q} \mapsto a^{i}_{q} \in \tilde{\Gamma}^{i}_{q}$.
    \end{theorem}

Here, the transition kernel for the MDP $(\pi_{q},(a^{1}_{q},...,a^{N}_{q}); q\geq 1)$ is given by 
\begin{align} \label{Theta definition}
    &\theta(\pi_{q+1}\in \cdot|\pi_{q},a_{q}^{[1,N]}):=P(\pi_{q+1}\in \cdot |\pi_{q},a_{q}^{[1,N]})\\
 \nonumber   &=\int \mathbbm{1}_{\{G(\pi_{q},u_{[qK,qK+K-1]}^{[1,N]},y_{[qK,qK+K-1]}^{[1,N]}) \in \cdot\}}\\
 &P(dy^{[1,N]}_{[qK,qK+K-1]}, du_{[qK,qK+K-1]}^{[1,N]}|\pi_{[0,q]},a_{[0,q]}^{[1,N]})   \end{align}
Where,
\begin{align*}
    &P(dy^{[1,N]}_{[qK,qK+K-1]},d u_{[qK,qK+K-1]}^{[1,N]}|\pi_{[0,q]},a_{[0,q]}^{[1,N]})= \\
    &\int \pi_{q}(dx_{Kq})(\ \prod\limits_{i=1}^{N}Q^{i}(dy^{i}_{Kq}|x_{Kq})f^{i}_{qK}(du^{i}_{qK}|y^{i}_{Kq}))\\
    &\mathcal{T}(dx_{Kq+1}|x_{Kq},u^{[1,N]}_{Kq})\big( \prod\limits_{i=1}^{N}Q^{i}(dy^{i}_{Kq+1}|x_{Kq+1})\\
    &f^{i}_{qK+1}(du^{i}_{qK+1}|y^{i}_{[Kq,Kq+1]},u^{i}_{qK})\big)\\
    &\mathcal{T}(dx_{Kq+2}|x_{Kq+1}, u^{[1,N]}_{Kq+1})...\\
    &\mathcal{T}(dx_{qK+K-1}|x_{qK+K-2}, u^{[1,N]}_{qK+K-2})\\
    &(\ \prod\limits_{i=1}^{N}Q^{i}(dy^{i}_{qK+K-1}|x_{qK+K-1})\\
    &f^{i}_{qK+K-1}(du^{i}_{qK+K-1}|y^{i}_{[Kq,Kq+K-1]},u^{i}_{[qK,qK+K-2]}))
\end{align*}
and $G$ is defined recursively as 
\begin{equation*}
 \begin{aligned}
     &G(\pi_{q},u_{[qK,qK+K-1]}^{[1,N]},y_{[qK,qK+K-1]}^{[1,N]})=\pi_{q+1}\\
     &= Z_{(q+1)K} =F(Z_{(q+1)K-1},u_{(q+1)K-1}^{[1,N]},y_{(q+1)K-1}^{[1,N]})\\
    &=F(F(Z_{(q+1)K-2},u_{(q+1)K-2}^{[1,N]},y_{(q+1)K-2}^{[1,N]}),\\
    &u_{(q+1)K-1}^{[1,N]},y_{(q+1)K-1}^{[1,N]}) 
 \end{aligned}   
\end{equation*}

\noindent Note that the equivalent cost criteria is given by
\begin{equation}
    J=\sum_{q=1}^{\infty} (\beta^{K})^{q} E^{\Tilde{\gamma}}[\Tilde{c} \label{mykperiod}(\pi_{q},a^{1}_{q},...,a^{N}_{q})]
\end{equation}
where, 
\begin{align}  \label{New k-step Cost}
   \nonumber &\Tilde{c}(\pi_{q},a^{1}_{q},...,a^{N}_{q})= \int \pi_{q}(dx_{qK}) \\
   \nonumber &\sum_{t=qK}^{t=(q+1)K-1} \beta^{t-qK} (\prod\limits_{i=1}^{N}Q^{i}(dy^{i}_{qK}|x_{qK})f^{i}_{qK}(du^{i}_{qK}|y^{i}_{qK}))\\ 
    &\mathcal{T}(dx_{qK+1}|x_{qK},\mathbf{u}_{qK})...\mathcal{T}(dx_{t}|x_{t-1},\mathbf{u}_{t-1})\\
    &(\prod\limits_{i=1}^{N}Q^{i}(dy^{i}_{t}|x_{t})f^{i}_{t}(du^{i}_{t }|y^{i}_{[qK,qK+r]},u^{i}_{[qK,qK+r-1]}))c(x_{t},\mathbf{u}_{t})
\end{align}
Next, we present a result on the weak-Feller regularity of the kernel $\theta$ which will be consequential for the convergence of quantized Q-learning (Section \ref{LearningQSectionInf}).

\begin{theorem}[Theorem 5.4 \cite{OMZSY-centralized-reduction}]\label{theorem6} Suppose that $\mathbb{X}$ is compact. Additionally, suppose the transition kernel $\mathcal{T}\left(x_{t+1}. \mid x_{t}=x, u_{t}=u\right)$ is continuous under total variation in $x$ and $u$, and $h(x, y)$ is continuous in $x$. Then the MDP $\left(\pi q ; \mathbf{a}_{q}\right)$ is weak-Feller.
\end{theorem}
Theorem \ref{theorem6} leads to the following existence result.
\begin{corollary}\label{existence KSDISP}
    Suppose that $\mathbb{X}$ is compact. Additionally, suppose the transition kernel $\mathcal{T}\left(x_{t+1}  . \mid x_{t}=x, u_{t}=u\right)$ is continuous under total variation in $x$ and $u$, and $h(x, y)$ is continuous in $x$. Then, an optimal solution to the KSPISP problem exists and an optimal policy is stationary.
\end{corollary}

\section{Sliding Memory Periodic Information Sharing Pattern and Near Optimality for K-step Periodic Sharing} \label{Finite memory approximations} \label{sliding}
When attempting to learn optimal solutions for the $K$-step periodic information sharing problem (\ref{k-step info},\ref{Cost k-step}) a prohibitive issue arises due to the expanding size of common information. We address this problem in this section.

\subsection{Sliding Memory Periodic Information Sharing Pattern}

We first obtain a useful supporting result. 

\begin{theorem}\label{theorem8}
Let $\hat{Z}_{q}=(\pi_{q-M},I^{M}_{q})$ where $I^{M}_{q}=\{y_{[(q-M)K,qK-1]}^{[1,N]},a_{[(q-M),q-1]}\}$. Then, $(\hat{Z}_{q},a_{q})$ forms a controlled MDP which is equivalent to the K-periodic belief sharing pattern problem introduced in Theorem \ref{centralized reduction KSPISP}.
\end{theorem}

We recall that 
\[\pi_{q-M}(.)=P(x_{(q-M)K}\in.| y^{[1,N]}_{[0,(q-M)K-1]}, u^{[1,N]}_{[0,(q-M)K-1]})\] We also note that since $a_{[(q-M),q-1]}$ consists of all the maps used by the agents to map their private measurements into actions is informationally equivalent to $\{y_{[(q-M)K,qK-1]}^{[1,N]},a_{[(q-M),q-1]}, u_{[(q-M)K,qK-1]}^{[1,N]}\}$.
Here we endow actions with the topology introduced in Definition \ref{top k-step}, probability measures with the weak topology, and measurements with the standard topology of the measurement spaces.

\noindent\textbf{Proof.} We first note that
\begin{eqnarray*}
    &&\pi_{q+1-M}=\\
    &&G(\pi_{q-M},u_{[(q-M)K,(q-M)K+K-1]}^{[1,N]},y_{[(q-M)K,(q-M)K+K-1]}^{[1,N]})\\
               &&=:\tilde{G}(\pi_{q-M},I^{M}_{q})
\end{eqnarray*}
Here, $\pi_{q+1-M}$ can be expressed as just a function of $ (\pi_{q-M},I^{M}_{q})$ because the measurements needed for the update are contained in $I^{M}_{q}$. Additionally, because we take the map-valued actions to be deterministic, $u_{[(q-M)K,(q-M)K+K-1]}^{[1,N]}$ can be generated using $a_{q-M}$ which is also included in $I^{M}_{q}$. It then follows that 
\begin{eqnarray*}
\nonumber    &&P\bigg((\pi_{q+1-M},I^{M}_{q+1}) \in .\mid \pi_{s-M}, I_{s}^{M}, a_{s} \text{ }s \leq q\bigg)\\
&&=\mathbbm{1}_{\{\tilde{G}(\pi_{q-M},I^{M}_{q}) \in .\}} \mathbbm{1}_{\{y_{[(q-M+1)K,qK+K-1]}^{[1,N]},a_{[q+1-M,q]}\}}\\
 && P\bigg(dy_{[qK,qK+K-1]}^{[1,N]} \mid \pi_{q-M}, I_{s}^{M}, a_{s} \text{ }s \leq q\bigg)
\end{eqnarray*}
Where,
\begin{eqnarray*}
\nonumber    &&P\bigg(dy_{[qK,qK+K-1]}^{[1,N]} \mid \pi_{q-M}, I_{s}^{M}, a_{q} \ \ s \leq q\bigg)\\
\nonumber    &&=\int P(dy_{[qK,qK+K-1]}^{[1,N]} \mid x_{qK}, a_{q})\\
&&P(dx_{qK} \mid \pi_{q-M}, I_{s}^{M}, a_{q} \text{ }s \leq q)\\
\nonumber    &&=\int P\bigg(dy_{[qK,qK+K-1]}^{[1,N]} \mid x_{qK}, a_{q}\bigg)\pi_{q}(dx_{qK})
\end{eqnarray*}
Here the last equality follows since, given $\{\pi_{q-M}, I_{q}^{M}, a_{q}\}$, one can perform Bayesian updates to obtain $\pi_{q}$.\\
Thus, we get
\begin{equation} \label{tilde theota}
    P(\hat{Z}_{q+1}\in.\mid \hat{Z}_{s}, a_{s} \ \ s\leq q)=:\tilde{\theta}(\hat{Z}_{q+1}\in. \mid \hat{Z}_{q}, a_{q})
\end{equation}
Additionally by considering, $r(\hat{Z}_{q}, a_{q})=\tilde{c}(\phi(\hat{Z}_{q}),a_{q})$ where $\phi$ maps $\hat{Z}_{q}$ to $\pi_{q}$ through recursive Bayesian updates, one obtains the desired cost equivalence.
\qed

Note that one can also express the cost $r(\hat{Z}_{q}, a_{q})$ as follows
\begin{eqnarray*}
\nonumber && r(\hat{Z}_{q}, a_{q}) =\int \pi_{q}(\pi_{q-M},I^{M}_{q})(dx_{qK})\\
&&\int \sum_{t=q}^{t=(q+1)K-1} \beta^{t-qK} (\prod\limits_{i=1}^{N}Q^{i}(dy^{i}_{qK}|x_{qK})f^{i}_{qK}(du^{i}_{qK}|y^{i}_{qK}))\\
&&\mathcal{T}(dx_{qK+1}|x_{qK},\mathbf{u}_{qK}) \ldots \\
&&\mathcal{T}(dx_{t}|x_{t-1},\mathbf{u}_{t-1})(\prod\limits_{i=1}^{N}Q^{i}(dy^{i}_{t}|x_{t-1})\\
\nonumber    &&f^{i}_{t}(du^{i}_{t }|y^{i}_{[qK,qK+r]},u^{i}_{[qK,qK+r-1]}))c(x_{t},\mathbf{u}_{t}) \text{, }
\end{eqnarray*}
where \begin{eqnarray*}
\pi_{q}(\pi_{q-M},I^{M}_{q})=\pi_{q}(\pi_{q-M},y^{[1,N]}_{[(q-M)K,qK-1]}, u^{[1,N]}_{[(q-M)K,qK-1]})
\end{eqnarray*}is the probability induced on $x_{qK}$ given $I^{M}_{q}$ and the fact that the filter at time $q-M$ is given by $\pi_{q-M}$. 

\begin{remark} \label{remark 5.1}
    By applying a similar proof to the one provided above, one can show that Theorem \ref{theorem8} also holds if $I_{q}^{M}$ is defined as $I_{q}^{M}=\{y_{[(q-M)K,qK-1]}^{[1,N]},u_{[(q-M)K,qK-1]}^{[1,N]}\}$. 
\end{remark}   

The result of Theorem \ref{theorem8} motivates the following information structure:
\begin{definition} The Sliding Memory periodic information sharing pattern consists of the following information structure: each Agent $i$ at time $t=qK+r$ has access to the common information: $I^{C}_{t}=\{\ y^{[1,N]}_{[(q-M)K,qK-1]}, u^{[1,N]}_{[(q-M)K,qK-1]} \}$ in addition to the private information $\{y^{i}_{[qK,qK+r]}, u^{i}_{[qK,qK+r-1]}\}$.
\end{definition}

The following is a critical supporting result.
\begin{lemma} \label{lemma for k-sliding window}
Let $\mu=P(dy_{[qK,qK+K-1]}\mid\pi_{q-M}, I_{q}^{M}, a_{q})$, and $\nu=P(dy_{[qK,qK+K-1]}\mid \pi^{*}, I_{q}^{M}, a_{q})$ for some fixed $\pi^{*}\in \mathcal{P}(\mathbb{X})$ such that $\pi_{q-M}\ll \pi^{*}$. Then, \[\| \mu-\nu \|_{TV} \leq \|\pi_{q}(\pi_{q-M},I^{M}_{q})-\pi_{q}(\pi^{*},I^{M}_{q})\|_{TV}.\] 
\end{lemma}

\begin{remark}
    Note that if the initial distribution on $x_{0}$ is absolutely continuous with respect to some measure $\zeta$ then the predictor at all time stages will be absolutely continuous with respect to $\zeta$. Hence, $\pi^{*}$ can be chosen without any knowledge of the true predictor by simply choosing $\pi^{*}$ so that $\zeta \ll \pi^{*}$. This will have implications for the near optimality of the sliding memory periodic information sharing pattern. 
\end{remark}

\noindent\textbf{Proof.} Let $f$ be a measurable function such that $\|f\|_{\infty}\leq 1$.
\begin{eqnarray*}
\nonumber    &&\bigg|\int f(y_{[qK,qK+K-1]}^{[1,N]}) d\mu-\int f(y_{[qK,qK+K-1]}^{[1,N]}) d\nu\bigg|\\
&&=\bigg|\int f(y_{[qK,qK+K-1]}^{[1,N]})P(dy_{[qK,qK+K-1]}^{[1,N]} \mid x_{qK}, a_{q})\\
 \nonumber   &&\quad \pi_{q}(\pi_{q-M},I^{M}_{q})(dx_{qK})\\
 && \quad-\int f(y_{[qK,qK+K-1]}^{[1,N]})P(dy_{[qK,qK+K-1]}^{[1,N]} \mid x_{qK},a_{q})\\
 &&\pi_{q}(\pi^{*},I^{M}_{q})(dx_{qK}) \bigg|\\
 \nonumber   &&= \bigg|\int\pi_{q}(\pi_{q-M},I^{M}_{q})(dx_{qK})\\
 &&\int  f(y_{[qK,qK+K-1]}^{[1,N]})P(dy_{[qK,qK+K-1]}^{[1,N]} \mid x_{qK}, a_{q})\\
\nonumber    &&\quad -\int\pi_{q}(\pi^{*},I^{M}_{q})(dx_{qK})\\
&&\int f(y_{[qK,qK+K-1]}^{[1,N]})P(dy_{[qK,qK+K-1]}^{[1,N]} \mid x_{qK},a_{q}) \bigg|\\
\nonumber    &&\leq \|\pi_{q}-\pi_{q}^{*}  \|_{TV}  
\end{eqnarray*}

\qed 
\begin{assumption} \label{WF sliding window MDP}
    The transition kernel $\tilde{\theta}$ as defined by (\ref{tilde theota}) is weak-Feller.
\end{assumption}
The following gives sufficient conditions for Assumption \ref{WF sliding window MDP} to hold.
\begin{proposition} \label{prop tilde theta WF}
    Suppose that $\mathbb{X}$ is compact. Additionally, suppose the transition kernel $\mathcal{T}\left(x_{t+1}. \mid x_{t}=x, u_{t}=u\right)$ is continuous under total variation in $x$ and $u$, and $h(x, y)$ is continuous in $x$. Then $\tilde{\theta}$ is weak-Feller.
\end{proposition}

Please see the Appendix for a proof. Let $(\hat{Z}^{*}_{q},a_{q})=((\pi^{*},I_{q}^{M}),a_{q})$ be an approximation of $(\hat{Z}_{q},a_{q})$ that always uses $\pi^{*}$ as the predictor. Here, again, $\pi^{*}\in \mathcal{P}(\mathbb{X})$ and at any time stage $q$ we have that the predictor for $(\hat{Z}_{q},a_{q})$ satisfies: $\pi_{q-M}\ll \pi^{*}$. We denote the kernel of $(\hat{Z}^{*}_{q},a_{q})$ by $\hat{\theta}$. Additionally, we endow the latter with the same cost functional $r$ and we denote its optimal value function with initialization $\hat{Z}^{*}=(\pi^{*},I^{M}_{0})$ as $J_{W}(\hat{Z}^{*})$. Additionally, the value function for $(\hat{Z}_{q},a_{q})$ with initialization $\hat{Z}=(\pi,I^{M}_{0})$ is denoted by $J(\hat{Z})$. Note that here it is assumed that at time $q=0$ the coordinators in both problems have access to the last $MK$ measurement and actions in the form of $I^{M}_{0}$ as well there respective predictors for the state at time $t=-MK$. 
\begin{theorem} \label{boundonperformance}
    Under Assumption \ref{WF sliding window MDP}, any K-step periodic information sharing pattern problem can be approximated by a sliding memory periodic information sharing pattern problem assuming the following predictor stability condition holds:
    \begin{align} \label{L_q}
    \nonumber    &L_{q}=\sup_{\gamma} E_{\pi}^{\gamma} [\| \pi_{q}(\pi_{q-M},y^{[1,N]}_{[(q-M)K,qK-1]}, u^{[1,N]}_{[(q-M)K,qK-1]}) \\
        &\qquad -\pi_{q}(\pi^*,y^{[1,N]}_{[(q-M)K,qK-1]}, u^{[1,N]}_{[(q-M)K,qK-1]})\|_{TV} ] \rightarrow 0 \text{,}
    \end{align}
    as $M\rightarrow \infty$ for all $q$.
\end{theorem}

We note that in \ref{boundonperformance} the supremum is over all stationary Markovian policies which are admissible for the MDP $(\hat{Z}_{q},a_{q})$ introduced in Theorem \ref{theorem8}.

\noindent\textbf{Proof.} 
Let $\tilde{\beta} = \beta^K$. We write 
\begin{eqnarray*}
\nonumber J(\hat{Z})&=&\min_{a} r(\hat{Z},a)+\tilde{\beta} \int J(\hat{Z}_{1})\tilde{\theta}(\hat{Z}_{1}|\hat{Z},a)\\
\nonumber J_{W}(\hat{Z}^{*})&=&\min_{a} r(\hat{Z}^{*},a)+\tilde{\beta} \int J_{W}(\hat{Z}_{1})\hat{\theta}(\hat{Z}_{1}|\pi^{*},I_{0}^{M},a)
\end{eqnarray*}
Then,
\begin{align*}
   &|J(\hat{Z})-J_{W}(\hat{Z}^{*})|\leq\max_{a}|r(\hat{Z},a)-r(\hat{Z}^{*},a)|\\
 &\quad \quad + \tilde{\beta}\max_{a}|E[J_{W}(\hat{Z}_{1})|(\pi^{*}),I_{0}^{M},a]\\
 & \quad \quad -E[J_{W}(\hat{Z}_{1})|(\pi),I_{0}^{M},a]|\\
   &\quad  \quad +\tilde{\beta} \max_{a}|E[J_{W}(\hat{Z}_{1})|(\pi),I_{0}^{M},a]\\
 & \quad \quad -E[J(\hat{Z}_{1})|\pi,I_{0}^{M},a]|
\end{align*}
We have that for all $a$
\begin{eqnarray*}
\nonumber        &&|r(\hat{Z},a)-r(\hat{Z}^{*},a)|=\bigg|\int \pi_{0}(\pi,I^{M}_{0})(dx_{0})\\
&&\int \sum_{t=0}^{t=K-1} \beta^{t} (\prod\limits_{i=1}^{N}Q^{i}(dy^{i}_{0}|x_{0})f^{i}_{0}(du^{i}_{0}|y^{i}_{0}))\\
 \nonumber       &&\quad\mathcal{T}(dx_{1}|x_{0},\mathbf{u}_{0})...\mathcal{T}(dx_{t}|x_{t-1},\mathbf{u}_{t-1})\\
 &&(\prod\limits_{i=1}^{N}Q^{i}(dy^{i}_{t}|x_{t-1})f^{i}_{t}(du^{i}_{t }|y^{i}_{[0,r]},u^{i}_{[0,r-1]}))c(x_{t},\mathbf{u}_{t})\\
\nonumber        &&\quad -\int \pi_{0}(\pi^{*},I^{M}_{0})(dx_{0})\\
&&\int \sum_{t=0}^{t=K-1} \beta^{t} (\prod\limits_{i=1}^{N}Q^{i}(dy^{i}_{0}|x_{0})f^{i}_{0}(du^{i}_{0}|y^{i}_{0})) \mathcal{T}(dx_{1}|x_{0},\mathbf{u}_{0})\\
  \nonumber      &&\quad ... \mathcal{T}(dx_{t}|x_{t-1},\mathbf{u}_{t-1})(\prod\limits_{i=1}^{N}Q^{i}(dy^{i}_{t}|x_{t-1})\\
  &&f^{i}_{t}(du^{i}_{t }|y^{i}_{[0,r]},u^{i}_{[0,r-1]}))c(x_{t},\mathbf{u}_{t})\bigg| \\
 \nonumber       &&\leq \frac{\|c\|_{\infty}}{1-\beta}\| \pi_{0}(\pi,y^{[1,N]}_{[-MK,-1]}, u^{[1,N]}_{[-MK,-1]})\\
 && \quad \quad  \quad \quad -\pi_{0}(\pi^*,y^{[1,N]}_{[-MK,-1]}, u^{[1,N]}_{[-MK,-1]})\|_{TV} 
\end{eqnarray*}
Similarly, we have that 
\begin{eqnarray*}
\nonumber    &&|E[J_{W}(\hat{Z}_{1})|(\pi^{*}),I_{0}^{M},a]-E[J_{W}(\hat{Z}_{1})|(\pi),I_{0}^{M},a]| \\
\nonumber    &&\leq \|J_{W}\|_{\infty} E\bigg [\|P(dy_{[0,K-1]}\mid \pi^{*}, I_{0}^{M}, a)\\
&& \qquad \qquad \qquad \qquad -P(dy_{[0,K-1]}\mid \pi, I_{0}^{M}, a)\|_{TV} \bigg]\\
\nonumber    &&\leq \frac{\|c\|_{\infty}}{1-\beta}E\bigg[\| \pi_{0}(\pi,y^{[1,N]}_{[-MK,-1]}, u^{[1,N]}_{[-MK,-1]}) \\
&& \qquad \qquad \qquad   -\pi_{0}(\pi^*,y^{[1,N]}_{[-MK,-1]}, u^{[1,N]}_{[-MK,-1]})\|_{TV}\bigg] 
\end{eqnarray*}
where the last inequality follows by Lemma \ref{lemma for k-sliding window}. Moreover,  
\begin{eqnarray*}
\nonumber &&|E[J_{W}(\hat{Z}_{1})|\pi,I_{0}^{M},a]-E[J(\hat{Z}_{1})|\pi,I_{0}^{M},a]|\\
&&\leq  E\bigg[|J_{W}(\hat{Z}_{1}^{*})-J(\hat{Z}_{1})| \bigg| \pi,I_{0}^{M},a\bigg ]   
\end{eqnarray*}
Thus we get that
\begin{eqnarray*}
\nonumber    &&|J(\hat{Z})-J_{W}(\hat{Z}^{*})|\leq L_{0}(\frac{\|c\|_{\infty}}{1-\beta}+\tilde{\beta}\frac{\|c\|_{\infty}}{1-\beta})\\
&& \quad +\tilde{\beta}\max_{a}E[|J_{W}(\hat{Z}_{1}^{*})-J(\hat{Z}_{1})|\mid \pi,I_{0}^{M},a]\\
&& \leq\frac{\|c\|_{\infty}}{1-\beta}\sum_{q=0}^{\infty}\tilde{\beta}^{q}L_{q}
\end{eqnarray*}
Where the last inequality follows by repeating the same process. \qed

Next we present a robustness result which shows that the optimal policy $\gamma^{*}$ for the MDP $(\hat{Z}^{*}_{q},a_{q})$ is near optimal for the MDP $\hat{Z}_{q},a_{q}$. In particular, the policy is applied as follows: for any $\hat{Z}_{q}$, $\gamma^{*}(\hat{Z}_{q})=\gamma^{*}(\hat{Z}^{*}_{q})$
\begin{theorem}
Let $J(\hat{Z}_{q},\gamma^{*})$ denote the performance attained the policy $\gamma^{*}$ is applied to the MDP $\hat{Z}_{q},a_{q}$. Then,
\begin{align*}
     &E^{\gamma^{*}}\bigg[ J(\hat{Z}_{0},\gamma^{*})-J(\hat{Z}_{0})\bigg ]\leq 2 \frac{\|c\|_{\infty}}{1-\beta}\sum_{q=0}^{\infty}\tilde{\beta}^{q}L_{q}
\end{align*}
\end{theorem}
\noindent\textbf{Proof sketch.} We have that \begin{align*}
    & E^{\gamma^{*}}\bigg[ J(\hat{Z}_{0},\gamma^{*})-J(\hat{Z}_{0})\bigg ]\leq  E^{\gamma^{*}}\bigg[ J(\hat{Z}_{0},\gamma^{*})-J_{W}(\hat{Z}^{*}_{0})\bigg] \\
     & \quad + E^{\gamma^{*}}\bigg[ J_{W}(\hat{Z}^{*}_{0})- J(\hat{Z}_{0})) \bigg ]
\end{align*}
Thus, by Theorem \ref{boundonperformance}, it is sufficient to show that $E^{\gamma^{*}}\bigg[ J(\hat{Z}_{0},\gamma^{*})- J_{W}(\hat{Z}^{*}_{0})\bigg] \leq \frac{\|c\|_{\infty}}{1-\beta}\sum_{q=0}^{\infty}\tilde{\beta}^{q}L_{q}$. We have that 
\begin{align*}
    &E^{\gamma^{*}}[ J(\hat{Z}_{0},\gamma^{*})]=r(\hat{Z}_{0},\gamma^{*}(\hat{Z}^{*}_{0}))+\tilde{\beta}\int J(\hat{Z}_{1},\gamma^{*})\tilde{\theta}(d\hat{Z}_{1}|\hat{Z},a) 
\end{align*}
Hence, 
\begin{align*}
   & E^{\gamma^{*}}\bigg[ J(\hat{Z}_{0},\gamma^{*})-J_{W}(\hat{Z}^{*}_{0})\bigg ] \leq \max_{a}|r(\hat{Z}_{0},a)-r(\hat{Z}^{*}_{0},a)|\\
 &\quad \quad + \tilde{\beta}\max_{a}|E^{\gamma^{*}}[J_{W}(\hat{Z}_{1})|(\pi^{*}),I_{0}^{M},a]\\
 & \qquad \qquad -E^{\gamma^{*}}[J_{W}(\hat{Z}_{1})|(\pi),I_{0}^{M},a]|\\
   &\quad  \quad +\tilde{\beta} \max_{a}|E^{\gamma^{*}}[J_{W}(\hat{Z}_{1})|(\pi),I_{0}^{M},a]\\
 & \qquad \qquad -E^{\gamma^{*}}[J(\hat{Z}_{1})|\pi,I_{0}^{M},a]|\\
 &\leq L_{0}(\frac{\|c\|_{\infty}}{1-\beta} +\tilde{\beta}\|J_{W}\|_{\infty})+ \tilde{\beta}\max_{a}|E^{\gamma^{*}}[J_{W}(\hat{Z}_{1})|(\pi),I_{0}^{M},a]\\
 & \quad \quad -E^{\gamma^{*}}[J(\hat{Z}_{1})|\pi,I_{0}^{M},a]|\\
 &\leq \frac{\|c\|_{\infty}}{1-\beta} \sum_{q=0}^{\infty}\tilde{\beta}^{q}L_{q}
\end{align*}





\subsection{Predictor stability under expected total variation} \label{predictor stablity}
In this section, we establish exponential predictor stability in expected total variation for the OSDISP problem and the KSPISP problem using Dobrushin coefficients. A similar result has been established for centralized POMDPs in \cite{mcdonald2020exponential}. Both predictor and filter stability play an important role in establishing robustness to an incorrectly designed policy (\cite{mcdonaldyuksel2022robustness}). In particular, it plays a crucial role in establishing the near optimality of finite window policies for POMDPs (\cite{kara2021convergence}). Similarly, as was seen in Section \ref{Finite memory approximations} predictor stability for decentralized problems plays a crucial role in demonstrating the near optimality of policies for which the common information consists of a sliding window of information.

 \begin{definition}
\cite{dobrushin1956central}: For a kernel operator $K:S_{1}\rightarrow \mathcal{P}(S_{2})$ we define the Dobrushin coefficient as: 
\begin{equation} \label{dobrushin coefficient}
    \delta(K)=\inf \sum_{i=1}^{n} \min(K(x,A_{i}),K(y,A_{i}))
\end{equation}
where the infimum is over all $x,y$ $\in S_{1}$ and all partitions $\{A_{i}\}_{i=1}^{n}$ of $S_2$. An equivalent and often useful definition of the Dobrushin coefficient is $\delta(K)=1-\alpha(K)$ where, 
\begin{equation*}
    \alpha(K)=\sup_{\mu \not = \nu} \frac{\|K(\mu)-K(\nu)\|_{TV}}{\|\mu-\nu\|_{TV}}
\end{equation*}
\end{definition}
A useful property, to be utilized, of the Dobrushin coefficient is that for any two probability measures $\mu,\nu$ $\in \mathcal{P}(S_{2})$: 
\begin{eqnarray*}
    \|K(\mu)-K(\nu)\|_{TV} \leq (1-\delta(K))\|\mu-\nu\|_{TV}
\end{eqnarray*}

\begin{theorem} \label{predictor stability one-step delayed theorem}
     Consider a true prior $\mu$ and a false prior $\nu$ with $\mu \ll \nu$. Then, the predictor is exponentially stable in total variation in expectation for the OSDISP AND KSPISP if the following holds: 
     \begin{eqnarray} \label{condition for predictor stability}
     (2-\delta(Q))(1-\tilde{\delta}(\mathcal{T}))<1
      \end{eqnarray}
     where, $Q(.|x_{t})=P(y_{t}^{[1,N]}\in.|x_{t})$, $\delta(Q)$ is the Dobrushin coefficient of the kernel $Q$, and $\tilde{\delta}(\mathcal{T})$ is the infimum of the Dobrushin coefficient of the kernel $\mathcal{T}(.|.,\mathbf{u})$ where the infimum is taken over all actions.
\end{theorem}

%

\textbf{Proof}. We first consider the OSDISP. Let $\alpha_{\mu}(\mathbf{y_{t}})=\int_{\mathbb{X}} \mu(dx_{t})\prod\limits_{i=1}^{N}h(x_{t},y_{t}^{i})$ and 
$\alpha_{\nu}(\mathbf{y_{t}})=\int_{\mathbb{X}} \nu(dx_{t})\prod\limits_{i=1}^{N}h(x_{t},y_{t}^{i})$.

\begin{align*}
\nonumber        &\|F(\mu,\mathbf{u}_{t},y_{t}^{[1,N]})-F(\nu,\mathbf{u}_{t},y_{t}^{[1,N]})\|_{TV}=\\
&\sup_{\|f\|_{\infty}\leq 1}\bigg|  \frac{\int_{\mathbb{X}} f(x_{t+1}) \mathcal{T}(dx_{t+1}|x_{t}, u_{t}^{[1,N]})\mu(dx_{t})\prod\limits\limits_{i=1}^{N} h(x_{t},y_{t}^{i})}{\alpha_{\mu}(\mathbf{y}_{t})}\\
 \nonumber       &\quad- \frac{\int_{\mathbb{X}} f(x_{t+1})\mathcal{T}(dx_{t+1}|x_{t}, u_{t}^{[1,N]})\nu(dx_{t})\displaystyle \prod\limits_{i=1}^{N} h(x_{t},y_{t}^{i})}{\alpha_{\nu}(\mathbf{y}_{t})} \bigg|\\
  \nonumber      & \leq (1-\tilde{\delta}(\mathcal{\tau}))\sup_{\|f\|_{\infty}\leq 1}\bigg|  \frac{\int_{\mathbb{X}} f(x_{t}) \mu(dx_{t})\prod\limits_{i=1}^{N} h(x_{t},y_{t}^{i})}{\alpha_{\mu}(\mathbf{y}_{t})}\\
  & \quad \quad  \quad \quad  - \frac{\int_{\mathbb{X}} f(x_{t})\nu(dx_{t})\prod\limits_{i=1}^{N} h(x_{t},y_{t}^{i})}{\alpha_{\nu}(\mathbf{y}_{t})} \bigg|\\
 & \leq  (1-\tilde{\delta}(\mathcal{\tau})) (T_{1}+T_{2})
\end{align*}

Where, 

\begin{align}
 \nonumber  &T_{1}=\sup_{\|f\|_{\infty}\leq 1}\bigg|  \frac{\int_{\mathbb{X}} f(x_{t})\mu(dx_{t})\prod\limits_{i=1}^{N} h(x_{t},y_{t}^{i})}{\alpha_{\mu}(\mathbf{y}_{t})}\\
 \nonumber &\qquad \qquad \quad \quad -\frac{\int_{\mathbb{X}} f(x_{t})\nu(dx_{t})\prod\limits_{i=1}^{N} h(x_{t},y_{t}^{i})}{\alpha_{\mu}(\mathbf{y}_{t})} \bigg|
\end{align} 

and 
\begin{align}
\nonumber    &T_{2}=\\
\nonumber &\sup_{\|f\|_{\infty}\leq 1}\big|\frac{\alpha_{\nu}(\mathbf{y}_{t})-\alpha_{\mu}(\mathbf{y}_{t})}{\alpha_{\nu}(\mathbf{y}_{t})\alpha_{\mu}(\mathbf{y}_{t})} \big| \bigg| \int_{\mathbb{X}} f(x_{t})\nu(dx_{t})\prod\limits_{i=1}^{N} h(x_{t},y_{t}^{i}) \bigg|
\end{align}
Next, we consider the expectation of $T_{1}$ and $T_{2}$ separately.

\begin{align*}
 & E^{\mu}[T_{1}] \\
 &=  \int \sup_{\|f\|_{\infty}\leq 1} \bigg| \frac{\int_{\mathbb{X}} f(x_{t})\mu(dx_{t})\prod\limits_{i=1}^{N} h(x_{t},y_{t}^{i})}{\alpha_{\mu}(\mathbf{y}_{t})}\\
 & \qquad\quad \qquad- \frac{\int_{\mathbb{X}} f(x_{t})\nu(dx_{t})\prod\limits_{i=1}^{N} h(x_{t},y_{t}^{i})}{\alpha_{\mu}(\mathbf{y}_{t})} \bigg|\\
 & \qquad \prod\limits_{i=1}^{N} h(x_{t},y^{i}_{t}) \prod\limits_{i=1}^{N}\psi(dy^{i}_{t})\mu(dx_{t})\\
 \nonumber   &= \int \sup_{\|f\|_{\infty}\leq 1} \bigg|\int_{\mathbb{X}} f(x_{t})\mu(dx_{t})\prod\limits_{i=1}^{N} h(x_{t},y_{t}^{i})\\
 & \qquad- \int_{\mathbb{X}} f(x_{t})\nu(dx_{t})\prod\limits_{i=1}^{N} h(x_{t},y_{t}^{i})\bigg| \prod\limits_{i=1}^{N}\psi(dy^{i}_{t})\\
  &= \int   f^{*}(x_{t})\mu(dx_{t})\prod\limits_{i=1}^{N} h(x_{t},y_{t}^{i}) \prod\limits_{i=1}^{N}\psi(dy^{i}_{t})\\
  & \qquad - \int f^{*}(x_{t+1})\nu(dx_{t})\prod\limits_{i=1}^{N} h(x_{t},y_{t}^{i})\prod\limits_{i=1}^{N}\psi(dy^{i}_{t})\\
  &= \int   f^{*}(x_{t})Q(d\mathbf{y}_{t}|x_{t})\mu(dx_{t})\\
  & \qquad - \int f^{*}(x_{t})Q(d\mathbf{y}_{t}|x_{t})\nu(dx_{t})\\
  \nonumber  &\leq \|\mu-\nu\|_{TV}
\end{align*}

\begin{align*}
\nonumber &E^{\mu}[T_{2}]\\
&=\int \sup_{\|f\|_{\infty}\leq 1}\big|\frac{\alpha_{\nu}(\mathbf{y}_{t})-\alpha_{\mu}(\mathbf{y}_{t})}{\alpha_{\nu}(\mathbf{y}_{t})\alpha_{\mu}(\mathbf{y}_{t})} \big| \\
&\bigg|  \int f(x_{t})\nu(dx_{t})\prod\limits_{i=1}^{N} h(x_{t},y_{t}^{i}) \bigg|\prod\limits_{i=1}^{n} h(x_{t},y^{i}_{t}) \prod\limits_{i=1}^{n}\psi(dy^{i}_{t})\mu(dx_{t})\\
 \nonumber   &=\int \sup_{\|f\|_{\infty}\leq 1}\big|\frac{\alpha_{\nu}(\mathbf{y}_{t})-\alpha_{\mu}(\mathbf{y}_{t})}{\alpha_{\nu} (\mathbf{y}_{t})}\big|\\
 & \bigg|\int f(x_{t}) \nu(dx_{t})\prod\limits_{i=1}^{N} h(x_{t},y_{t}^{i}) \bigg|  \prod\limits_{i=1}^{n}\psi(dy^{i}_{t})\\
 \nonumber   &\leq\int \sup_{\|f\|_{\infty}\leq 1}|\frac{\alpha_{\nu}(\mathbf{y}_{t})-\alpha_{\mu}(\mathbf{y}_{t})}{\alpha_{\nu}(\mathbf{y}_{t}) }| \int \nu(dx_{t})\prod\limits_{i=1}^{N} h(x_{t},y_{t}^{i}) \prod\limits_{i=1}^{n}\psi(dy^{i}_{t})\\   
  \nonumber  &\leq \int |\alpha_{\mu}(\mathbf{y}_{t})-\alpha_{\nu}(\mathbf{y}_{t})|\prod\limits_{i=1}^{n}\psi(dy^{i}_{t})
\end{align*}
Let $S^{+}=\{\mathbf{y}\in \prod\limits_{i=1}^{N}\mathbb{Y}^{i}\mid \alpha_{\mu}(\mathbf{y})-\alpha_{\nu}((\mathbf{y}))>0\}$ and $S^{-}=\{\mathbf{y}\in \prod\limits_{i=1}^{N}\mathbb{Y}^{i}\mid \alpha_{\mu}(\mathbf{y})-\alpha_{\nu}(\mathbf{y})<0\}$. Then, we get

\begin{align*}
\nonumber    &\int |\alpha_{\mu}(\mathbf{y}_{t})-\alpha_{\nu}(\mathbf{y}_{t})|\prod\limits_{i=1}^{N}\psi(dy^{i}_{t})\\
&=\int_{S^{+}} (\alpha_{\mu}(\mathbf{y}_{t})-\alpha_{\nu}(\mathbf{y}_{t}))\prod\limits_{i=1}^{n}\psi(dy^{i}_{t})\\
& \quad \quad \quad +\int_{S^{-}} (\alpha_{\nu}(\mathbf{y}_{t})-\alpha_{\mu}(\mathbf{y}_{t}))\prod\limits_{i=1}^{n}\psi(dy^{i}_{t})\\
 \nonumber   &=\int \mu(dx_{t}) \int (\mathbbm{1}_{S^{+}}(\mathbf{y}_{t})-\mathbbm{1}_{S^{-}}(\mathbf{y}_{t}))\prod\limits_{i=1}^{N}h(x_{t},y_{t}^{i})\prod\limits_{i=1}^{n}\psi(dy^{i}_{t}) \\
 & -  \int \nu(dx_{t}) \int (\mathbbm{1}_{S^{+}}(\mathbf{y}_{t})-\mathbbm{1}_{S^{-}}(\mathbf{y}_{t}))\prod\limits_{i=1}^{N}h(x_{t},y_{t}^{i})\prod\limits_{i=1}^{n}\psi(dy^{i}_{t})  \\
 &\leq \sup_{\|f\|_{\infty}\leq 1} \big|\int f(\mathbf{y}_{t}) Q(d\mathbf{y}_{t}|x_{t})\mu(dx_{t}) \\
\nonumber & \qquad \qquad \qquad \qquad -\int f(\mathbf{y}_{t}) Q(d\mathbf{y}_{t}|x_{t})\nu(dx_{t})\big|\\
 & = (1-\delta(Q))\|\mu-\nu\|_{TV}
\end{align*}

\qed

\begin{remark}
It is important to note that the previous inequality is independent of policy hence we obtain a bound that is uniform over all policies.
\end{remark}

Now, for the KSPISP, since the predictor update is obtained by applying the Bayesian update for the one step delayed problem $K$-times, the desired result follows. Note that here once again, the bound obtained is independent of policy thus allowing us to obtain a uniform bound for $L_{q}$ as defined in (\ref{L_q}).\qed 

\subsection{Predictor stability via the Hilbert metric}
In this section we provide another sufficient condition for (\ref{L_q}) to hold.

    \begin{definition}
Let $\mu$ and $\nu$ be two non-negative finite measures on a measurable space $(\mathbb{X},\mathcal{F})$. Define
\[
h(\mu,\nu) :=
\begin{cases}
\displaystyle
\log\!\left(
\frac{\sup\limits_{A \in \mathcal{F}:\,\nu(A)>0} \frac{\mu(A)}{\nu(A)}}
     {\inf\limits_{A \in \mathcal{F}:\,\nu(A)>0} \frac{\mu(A)}{\nu(A)}}
\right),
& \text{if $\mu \sim \nu$ }, \\[2.5ex]
0, & \text{if } \mu = \nu \equiv 0, \\[1ex]
+\infty, & \text{otherwise}.
\end{cases}
\]
\end{definition}
A useful property of the Hilbert metric is that 
\begin{equation} \label{hilbert inequality}
    \|\mu-\nu\|_{TV}\leq \frac{2}{3}h(\mu,\nu)
\end{equation}

\begin{assumption} \label{assump HM}
    (i) $\mathbb{Y}$, $\mathbb{U}$ are finite. (ii) For all $x\in \mathbb{X}$, $\mathbf{y}\in \mathbb{Y}$: $Q(\mathbf{y}|x)\geq \epsilon$ for some $\epsilon>0$. (ii) There exists some nonnegative measure $\lambda$ on $\mathbb{X}$, and a constant $\epsilon_{\mathbf{u}}>0$ such that for any Borel $A$ and any $x\in \mathbb{X}$, we have \[\epsilon_{\mathbf{u}}\lambda(A)\leq \tau(A|x,\mathbf{u}) \leq \frac{1}{\epsilon_{\mathbf{u}}} \lambda(A)\]
\end{assumption}
\begin{theorem}\label{PredStabHilbert}
    Under Assumption \ref{assump HM}, condition (\ref{L_q}) holds.
\end{theorem}
\noindent\textbf{Proof sketch.} It follows from \cite[Proposition 3.9]{le2004stability} or \cite[Lemma 5]{demirciRefined2023} that the predictor is path-wise geometrically stable under the Hilbert metric thus by the dominated convergence theorem it is geometrically stable in expectation under the Hilbert metric. The result then follows from inequality (\ref{hilbert inequality}). \qed

\section{Q-learning: Convergence to Near Optimality under Decentralized Information Structures}\label{LearningQSectionInf}

So far in the paper, we developed general approximation results. However, some of these may be difficult to implement numerically (due to the Bayesian updates involved), not unlike similar results for POMDPs. Accordingly, rigorously justified reinforcement algorithms would be desirable for these settings. 

With this motivation, in this section we first provide a learning algorithm for the KSPISP, building on \cite{KSYContQLearning} for weak-Feller MDPs. We note that a related learning algorithm has been studied in \cite{mao2023decentralized} where the authors provide a MARL (Multi-agent reinforcement learning) algorithm for the case where the agents all have finite memory and policies are restricted to a parametric class which may not include all admissible policies. The authors show that the algorithm converges in finite time and numerically demonstrated its performance. In contrast, in our work we rigorously establish convergence to near optimality and we allow for more general setups. In particular, we provide numerical learning algorithms which build upon the centralized reduction of the KSPISP, discretization of the space of predictors, as well as our rigorous approximation results.   
\subsection{Q-learning via predictor quantization (QPQ)}
\subsubsection{Description of the quantization procedure and the Q-learning algorithm}
Suppose for some $n \in \bf{N}$ we quantize the space of predictors $\mathcal{P}(X)$ into multiple bins $B_{1},...,B_{m}$ such that $\displaystyle \sup_{\pi \in B_{i}}W_{1}(\pi,\pi_{i})\leq \frac{1}{n}$ for all $i \in \{1,...,m\}$. Where, $\pi_{i}$ is a representative element from the bin $B_{i}$. Additionally, with $\bigcup_{i=1}^{m}B_{i}=\mathcal{P}(\mathbb{X})$, we apply policies that are constant over each bin. When the MDP transitions to a new state, it is approximated by a fixed element called a "quantizer output" of the bin to which it is the closest. We will use the notation $\pi^{c}_{i}$ to refer to the quantizer outpu or representative element of bin $i$. Let $\displaystyle q(\pi'):=\arg\min_{\{\pi \in \{\pi^{c}_{1},...,\pi^{c}_{m}\}\}}W_{1}(\pi',\pi)$. If the minimum is achieved by two different points, it is assumed that the selection is made in whichever way which guarantees the measurability of $q$. Consider the following assumption on the learning rates.
\begin{assumption} \label{specififc learning rates}
    For all $\pi \in \{\pi^{c}_{1},...,\pi^{c}_{m}\}$, $a\in \mathbb{A}$, and for all $k\geq 1$, we have
    \begin{enumerate}        
        \item $\alpha_{k}(\pi,a)=0$ unless $(\pi,a)=(\pi_{k},a_{k})$.
        \item  $\displaystyle \alpha_{k}(\pi,a)=\frac{1}{1+\sum_{l=0}^{k} \mathbbm{1}_{(\pi_{l},a_{l})=(\pi,a)}}$
    \end{enumerate}
\end{assumption}
Under Assumption \ref{specififc learning rates}, Algorithm \ref{alg:AQPQ} is constructed. 

{\setlength{\topsep}{0pt}
 \setlength{\partopsep}{0pt}
 \setlength{\parskip}{0pt}
\begin{algorithm}[h]

\caption{Quantized Asynchronous Q-learning (AQPQ)}
\label{alg:AQPQ}
\begin{algorithmic}[1]

\State Arbitrarily initialize $Q_{0}$
\State Initialize $\pi_{0} \sim \mu$

\For{$t = 0,1,\dots,T$}
    \State Generate an action using the exploration policy $\gamma$: 
        $a_{t} \sim \gamma(\,\cdot \mid \pi_{t})$
    \State Generate next state: 
        $\pi_{t+1} \sim \theta(\,\cdot \mid \pi_{t}, a_{t})$
    \State Quantize the state: $\pi = q(\pi_{t})$

    \ForAll{$(\pi,a)$}
        \If{$(\pi,a) = (\pi_{t}, a_{t})$}
            \State 
            \begin{align*}
            &Q_{t+1}(\pi,a)
              = \left(1-\alpha_{t}(\pi,a)\right) Q_{t}(\pi,a)\\         & \quad \quad \quad \quad \quad+ \alpha_{t}(\pi,a)
              \left(
              \tilde{c}(\pi_{t},a)
              + \tilde{\beta}\min_{v} Q_{t}(q(\pi_{t+1}),v)
              \right)
            \end{align*}
        \Else
            \State $Q_{t+1}(\pi,a) = Q_{t}(\pi,a)$
        \EndIf
    \EndFor
\EndFor

\State \Return $Q_{T}$

\end{algorithmic}
\end{algorithm}
}
Next, we present the following assumption which is suitable for a synchronous algorithm. 
\begin{assumption} \label{General learning rates}
    For all $k\geq 1$, we have
    \begin{enumerate}       
    \item $\alpha_{k}\in [0,1]$. $\sum_{k}\alpha_{k}=\infty$.
    \item $\sum_{k}\alpha^{2}_{k}\leq C$ for some constant $C<\infty$.
    \end{enumerate}
\end{assumption}
Under Assumption \ref{General learning rates}, we consider a  synchronous quantized Q-learning algorithm (Algorithm \ref{alg:SQPQ}). 

\vspace{-9pt}
\begin{algorithm}[H]
\caption{Quantized Synchronous Q-learning (SQPQ)}
\label{alg:SQPQ}
\begin{algorithmic}[1]

\State Arbitrarily initialize $Q_{0}$

\For{$t = 0,1,\dots,T$}
    \For{$\pi \in \{\pi^{c}_{1}, \dots, \pi^{c}_{m}\}$}
        \For{$a \in \mathbb{A}$}
            \State Generate next state: 
                $\pi_{t+1} \sim \theta(\,\cdot \mid \pi, a)$
            \State Quantize the state:
                \[
                \pi_{t+1}
                = \arg\min_{\pi \in \{\pi^{c}_{1},...,\pi^{c}_{m}\}} 
                W_{1}(\pi_{t+1}, \pi)
                \]

            \State 
            \begin{align*}
                   & Q_{t+1}(\pi,a)
            = \left(1 - \alpha_{t}\right) Q_{t}(\pi,a)\\             &\quad \quad \quad \quad \quad+ \alpha_{t}
            \left(
            \tilde{c}(\pi,a)
            + \tilde{\beta}
            \min_{v} Q_{t}(\pi_{t+1}, v)
            \right)
            \end{align*}                    
        \EndFor
    \EndFor
\EndFor

\State \Return $Q_{T}$

\end{algorithmic}
\end{algorithm}

\subsection{Convergence of QPQ}

\begin{assumption} \label{compactness and continuity of c} Suppose the kernel $\theta$ (\ref{Theta definition}) is weakly continuous. 
\end{assumption}
Note that sufficient conditions for Assumption \ref{compactness and continuity of c} to hold are provided in Theorem \ref{theorem6}.
\begin{assumption} \label{unique invarian probability measure}
    Assume that the exploration policy induces a unique invariant probability measure on the space of predictors $\mathcal{P}(\mathbb{X})$ and that the initialization $\mu$ is in the support of the invariant probability measure. Additionally, the unique invariant probability measure assigns measure zero to all the bin boundaries ($\partial B_{i}$ for $i=\{1,\ldots,m\}$).
\end{assumption}
\begin{remark}
    We note that a sufficient condition for the predictor process to admit a unique invariant probability measure under the uniform exploration policy is given by the following: (i) the kernel $\theta$ is weak-Feller. (ii) $\mathbb{X}$ is compact. (iii) The uniform exploration policy induces a unique invariant probability measure on the process $\{x_{t}\}_{t\geq 1}$. (iv) The predictor is stable in total variation. (see \cite[Theorem 4 ]{creggZeroDelayNoiseless} as well as \cite[Corollary 3]{DiMasiStettner2005ergodicity}). 
\end{remark}

\begin{theorem} \label{Convergence of Quantized Q-learning}
Consider the MDP $(\pi_{t},a_{t})$. Then, under Assumption \ref{compactness and continuity of c} and \ref{unique invarian probability measure},  the AQPQ algorithm (Algorithm 1) converges almost surely asymptotically to $Q^{*}(\pi,a)$. Additionally, a near optimal policy can be obtained by $\displaystyle \gamma^{*}(\pi)=\arg\min_{a\in \mathbb{A}}Q^{*}(q(\pi),a)$ where $\mathbb{A}$ is as defined in Section \ref{K periodic section}. 
\end{theorem}

\noindent\textbf{Proof.}: Since the state space, and the action space are compact, $\tilde{c}$ is continuous and bounded, and $\theta$ is weakly continuous, it follows from (\cite{KSYContQLearning} Theorem 7) and (\cite{SaLiYuSpringer} Theorem 4.3), that the AQPQ algorithm converges asymptotically to a near optimal solution for all the quantized states which are visited infinitely often during the exploration process. By Assumption \ref{unique invarian probability measure} all the states visited during the exploration process are visited infinitely often (see for example Theorem 5.1 in \cite{creggZeroDelayNoiseless}) and thus the convergence of Algorithm 1 to a near optimal solution applies to all states visited during the exploration phase.  \qed

We note that it follows from Corollary 3.2 \cite{karayukselNonMarkovian}, that the existence of a unique invariant probability measure on the predictor process assumption, could be replaced with the existence of an ergodic measure. However, sufficient conditions, for the existence of an ergodic measure on the predictor process, could be difficult to characterize. 
\begin{theorem}
    Under Assumption \ref{General learning rates}, the SQPQ algorithm (Algorithm 2) converges almost surely asymptotically to $Q^{*}(\pi,a)$ such that $\displaystyle \min_{a}Q^{*}(\pi,a)$ is a near optimal value function as the step size of the quantization ($n$) becomes sufficiently large. Additionally, a near optimal policy can be obtained by $\gamma^{*}(\pi)=\arg\min_{a\in \mathbb{A}}Q^{*}(q(\pi),a)$.
\end{theorem}

\noindent\textbf{Proof.} Here the proof follows from the near optimality of the quantized model \cite[Theorem 4.4]{OMZSY-centralized-reduction} as well as the convergence of the standard Watkins Q-learning algorithm \cite{Watkins}. \qed 

\subsection{Q-learning with sliding window periodic sharing pattern}
Now under Assumption \ref{specififc learning rates} we present a Q-learning algorithm with sliding window periodic sharing pattern (QSWPSP) (\ref{alg:QSWPSP}). Let $M$ denote the length of the sliding window of common information as defined in Theorem \ref{theorem8} so that $I^{M}_{qK}=\{y_{[(q-M)K,qK-1]}^{[1,N]}, u_{[(q-M)K,qK-1]}^{[1,N]}\}$. Here we assume that all measurement spaces $\mathbb{Y}^{1}$,...,$\mathbb{Y}^{N}$ are finite.

\begin{algorithm}[h]
\caption{QSWPSP}
\label{alg:QSWPSP}
\begin{algorithmic}[1]

\State Arbitrarily initialize $Q_{0}$

\For{$t = 0,\dots,KM$}
    \State Generate an action $a_{t}$ arbitrarily at random
\EndFor

\For{$q = M,\dots,T$}

    \State Generate an action $a_{q}$ arbitrarily at random
    \State Observe 
        $\{\mathbf{y}_{[qK,(q+1)K)]}, \mathbf{u}_{[qK,(q+1)K)]}\}$ 
        and the realized cost 
        \[
        \tilde{c}_{q}
        = \beta^{-qK}
          \sum_{t=qK}^{(q+1)K}
          c(x_{t}, \mathbf{u}_{t})
        \]
    \State Update the state 
        $I^{M}_{qK} \rightarrow I^{M}_{(q+1)K}$
        based on 
        $\{\mathbf{y}_{[qK,(q+1)K)]}, \mathbf{u}_{[qK,(q+1)K)]}\}$

    \ForAll{$(I^{M},a)$}
        \If{$(I^{M},a) = (I^{M}_{q}, a_{q})$}
            \State
            \begin{align*}
               & Q_{q+1}(I^{M},a)
              = \left(1-\alpha_{t}(I^{M},a)\right) Q_{q}(I^{M},a)\\
              &\quad \quad \quad \quad \quad + \alpha_{t}(I^{M},a)
                \left(
                  \tilde{c}_{q}
                  + \tilde{\beta}
                    \min_{v} 
                    Q_{q}\big(I^{M}_{(q+1)K}, v\big)
                \right)
            \end{align*}
        \Else
            \State $Q_{q+1}(I^{M},a) = Q_{q+1}(I^{M},a)$
        \EndIf
    \EndFor

\EndFor

\State \Return $Q_{T}$

\end{algorithmic}
\end{algorithm}

\begin{assumption} \label{PHR}
    The process $x_{t}$ induced by selecting actions arbitrarily at random in Algorithm QSWPSP admits a unique invariant probability measure, and every possible realization $I^{M}_{qK}$ under fixed window length $M$ is asymptotically realized infinitely often.
\end{assumption}
\begin{theorem}
    Suppose Assumption \ref{WF sliding window MDP} and condition \ref{condition for predictor stability} hold. Under Assumption \ref{PHR}, the QSWPSP Algorithm converges almost surely to a fixed point $Q^{*}(I^{M},a)$ as $T\rightarrow \infty$. Moreover, if the length of the window $M$ is sufficiently large the policy $\gamma(I^{M})=\arg\min_{a\in \mathbb{A}} Q^{*}(I^{M},a)$ is near optimal for the KSPISP.
\end{theorem}
\noindent\textbf{Proof.} The almost sure convergence of Algorithm 3 to a fixed point which corresponds to the optimal value function for the MDP $(\hat{Z}^{*},a_{q})$ follows from \cite[Theorem 8]{KSYContQLearning} whereas the near optimality of learned policy $\gamma(I^{M})=\arg\min_{a\in \mathbb{A}} Q^{*}(I^{M},a)$ for the KSPISP stems from Theorem \ref{boundonperformance}. \qed

\begin{remark}[Comparison of the Algorithms]
\begin{itemize}
\item[(i)] Algorithm \ref{alg:AQPQ} (AQPQ) requires system knowledge for Bayesian updates during the implementation as well as predictor stability in total variation (to ensure the existence of an ergodic invariant probability measure).
\item[(ii)] Algorithm \ref{alg:SQPQ} (SQPQ) leads to faster convergence, but requires access to a simulation device, in addition to system knowledge for Bayesian updates. However, it does not require predictor stability (due to the synchronous nature and access to simulation outcomes).
\item[(iii)] Algorithm \ref{alg:QSWPSP} (QSWPSP) is completely model free as it does not require the computation of Bayesian updates. However, it requires (geometric) predictor stability in expected total variation. 
An important benefit of the sliding window approximation introduced in Section \ref{Finite memory approximations} is that apriori knowledge of the support of the invariant probability measure induced by the exploration policy is not needed to determine an appropriate initialization, unlike AQPQ, as long as under such policy all possible values of measurements occur infinitely often in the case where the measurement spaces are finite. 
\end{itemize}
\end{remark}

\section{Numerical Examples}
Here, we will present several numerical examples to demonstrate the findings of this paper. In particular, in Section \ref{K step numerics} we present an application of Q-learning to two different problems under the KSPISP. In Section \ref{sliding mem numerics}, we present a numerical solution to an example under KSPISP with a finite length sliding window of common information.  

\subsection{$K$-step periodic information sharing pattern} \label{K step numerics}
\subsubsection{Example I}  Consider $\mathbb{X}=\{0,1,2\}$, $\mathbb{U}^{1}=\mathbb{U}^{2}=\mathbb{Y}^{1}=\mathbb{Y}^{2}=\{0,1\}$; $Q^{1}=Q^{2}=\begin{pmatrix}
    0.5 & 0.5\\
    0 & 1\\
    0 & 1\\
\end{pmatrix}$. Let 
\begin{eqnarray*}
&&\mathcal{T}(\cdot | \cdot) = 
\begin{pmatrix}
   \frac{1}{2} & \frac{1}{2} & 0\\
   \frac{1}{5} & \frac{3}{5} & \frac{1}{5}\\
   \frac{1}{3} & \frac{1}{3} & \frac{1}{3}
\end{pmatrix}
\mathbbm{1}_{\{ u^{1}_{t}=0, u^{2}_{t}=0\}}  \\
&&\quad  + \begin{pmatrix}
   \frac{1}{4} & \frac{1}{2} & \frac{1}{4}\\
   \frac{1}{6} & \frac{2}{3} & \frac{1}{6}\\
   \frac{2}{5} & \frac{2}{5} & \frac{1}{5}
\end{pmatrix} \mathbbm{1}_{\{ u^{1}_{t} =0, u^{2}_{t}=1\}}\\
&& + \begin{pmatrix}
   \frac{1}{2} & 0 & \frac{1}{2}\\
   \frac{1}{4} & \frac{1}{4} & \frac{1}{2}\\
   \frac{1}{2} & \frac{1}{4} & \frac{1}{4}
\end{pmatrix}
\mathbbm{1}_{\{ u^{1}_{t}=1, u^{2}_{t}=0\}}  + \begin{pmatrix}
   \frac{1}{3} & \frac{1}{3} & \frac{1}{3}\\
   \frac{1}{2} & \frac{1}{2} & 0\\
   \frac{1}{5} & \frac{3}{5} & \frac{1}{5}
\end{pmatrix}
\mathbbm{1}_{\{ u^{1}_{t}=1, u^{2}_{t}=1\}}
\end{eqnarray*} with $\beta=0.01$, under the KSPISP with $K=2$, consider the following cost functional $c(x,u_{1},u_{2})=(x-u_{1}-u_{2})^{2} + u_{1}^{2}  +u_{2}^2.$ We first quantize the spaces of predictors by quantizing $[0,1]^{3}$ with a step-size of $0.1$ and then normalizing every vector and pruning duplicates. We then estimate the transition kernel $\theta$ using $3\times 10^{5}$ samples for every state action pair. A benefit of this approach is that it allows for parallelization over each state with no communication overhead. We note that such an empirical model learning is equivalent to quantized $Q$-learning as discussed in \cite{zhou2024robustness},\cite{bicer2025quantizer}. Next, using the estimated kernel, we apply algorithm 2 as well as value iteration. In particular, we use $\alpha_{t}=\frac{1}{(1+t)^{0.85}}$ which is a choice motivated by \cite{evendar}. In Table \ref{tab:finite-spaces}, the first column represents the initial distribution $\mu_{0}$ such that $x_{0} \sim \mu_{0}$. The second column represents the value obtained by averaging the performance achieved in the original problem when the optimal policies obtained from the final Q-matrix are used. In particular the average is taken over $10^{5}$ samples and the infinite horizon cost criteria is approximated with one whereby $t\in\{0,\cdots,10\}$. The third column represents the optimal value function derived from Q-learning whereas the last column is the optimal value derived from value iteration.

\begin{table}[!htbp]
\centering
\caption{Example I}
\label{tab:finite-spaces}

\begin{tabular}{|l|l|l|l|}
\hline
\textbf{Initial distribution}           & \textbf{$J^{*}_{Sim}$} & \textbf{$J^{*}_{Q}$} & \textbf{$V^{*}$} \\ \hline
{[}0.3333	0.3333	0.3333{]}             & 1.3458                & 1.3441               & 1.3441                 \\ \hline
{[}0.3889	0.1667	0.4444{]} & 1.4519                & 1.4552                & 1.4552                 \\ \hline
{[}0.3889	0.2222	0.3889{]}    & 1.4023                & 1.3997              & 1.3997                \\ \hline
{[}0.3889	0.3889	0.2222{]} & 1.2352                & 1.2329
&1.2329                    \\ \hline 
{[}0.3889	0.4444	0.1667{]} & 1.1287                & 1.1218              & 1.1218                \\ \hline
{[}0.3913	0.2174	0.3913{]} & 1.4006                & 1.4021             & 1.4021                \\ \hline
{[}0.3913	0.3913	0.2174{]}  & 1.2273                 & 1.2281              & 1.2281               \\ \hline
{[}0.4	0.1	0.5{]}     & 1.5117                & 1.5108               & 1.5108                \\ \hline
\end{tabular}
\end{table}

    It is worth noting that the kernel $\theta$ is such that only a fraction ($<10\%$) of all elements in our quantized predictor space have a positive transition probability given some action and predictor. We consider such predictors to be the \textbf{effective} states as all other predictors do not occur after some finite time regardless of the initial distribution.  

Figure \ref{fig:Q_plot_ex2} shows the convergence of the Q-learning algorithm.

\noindent
\begin{minipage}{\linewidth}

    \centering
    \includegraphics[width=0.8\linewidth]{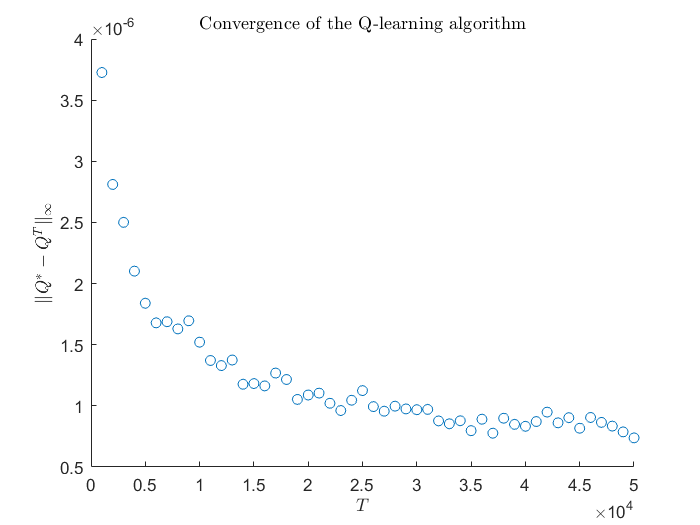}
    \captionof{figure}{Convergence of the quantized Q-learning algorithm for Example I}
    \label{fig:Q_plot_ex2}
\end{minipage}
 Next, we consider the following example.
\subsubsection{Example II} Consider $\mathbb{X}=\{0,1,2\}$, $\mathbb{U}^{1}=\mathbb{U}^{2}=\mathbb{Y}^{1}=\mathbb{Y}^{2}=\{0,1\}$. 
\begin{align*}
&Q^{1}=Q^{2}=\begin{pmatrix}
    0.5 & 0.5\\
    0 & 1\\
    0 & 1\\
\end{pmatrix}.\\
&\mathcal{T}(\cdot | \cdot) = 
\begin{pmatrix}
   \frac{1}{3} & \frac{1}{3} & \frac{1}{3}\\
   \frac{1}{3} & \frac{1}{3} & \frac{1}{3}\\
   \frac{1}{3} & \frac{1}{3} & \frac{1}{3}
\end{pmatrix} 1_{\{ u^{1}_{t}\not = u^{2}_{t}\}} \\
& \qquad \quad \quad + \begin{pmatrix}
   0 & 0.5 & 0.5\\
   0.5 & 0 & 0.5\\
   0.5 & 0.5 & 0
\end{pmatrix} 1_{\{ u^{1}_{t} = u^{2}_{t}\}}\\
& c(x,u_{1},u_{2})=(x-u_{1}-u_{2})^{2} + u_{1}^{2}  +u_{2}^2
\end{align*}
and $\beta=0.01$ with $K=2$.

Similar to the last example, we first provide a table which shows the results of Q-learning, value iteration, and performance in terms of the original model for various initial distributions (Table \ref{tab:example_II}). 

\begin{table}[H]
    \centering
    \caption{Example II}
    \label{tab:example_II}
    \begin{tabular}{|l|l|l|l|l|l|}
        \hline
        \textbf{Initial distribution} & \textbf{$J^{*}_{Sim}$} & \textbf{$J^{*}_{Q}$} & \textbf{$V^{*}$} \\ \hline
        {[}0	0.5	0.5{]}                   & 1.5073 & 1.5135 & 1.5135 \\ \hline
        {[}0.33333	0.33333	0.33333{]}     & 1.3645 & 1.3471 & 1.3471 \\ \hline
        {[}0.38889	0.16667	0.44444{]}     & 1.4528 & 1.4582 & 1.4582 \\ \hline
        {[}0.38889	0.22222	0.38889{]}     & 1.4308 & 1.4027 & 1.4027 \\ \hline
        {[}0.38889	0.38889	0.22222{]}     & 1.1845 & 1.236  & 1.236  \\ \hline
        {[}0.38889	0.44444	0.16667{]}     & 1.0952 & 1.1254 & 1.1254 \\ \hline
        {[}0.3913	0.21739	0.3913{]}       & 1.4102 & 1.4051 & 1.4051 \\ \hline
        {[}0.3913	0.3913	0.21739{]}       & 1.2239 & 1.2312 & 1.2312 \\ \hline
        {[}0.4	0.1	0.5{]}                & 1.5297 & 1.5138 & 1.5138 \\ \hline
        {[}0.4	0.15	0.45{]}              & 1.4276 & 1.4638 & 1.4638 \\ \hline
    \end{tabular}
\end{table}


\noindent
\begin{minipage}{\linewidth}
    \centering
    \includegraphics[width=0.9\linewidth]{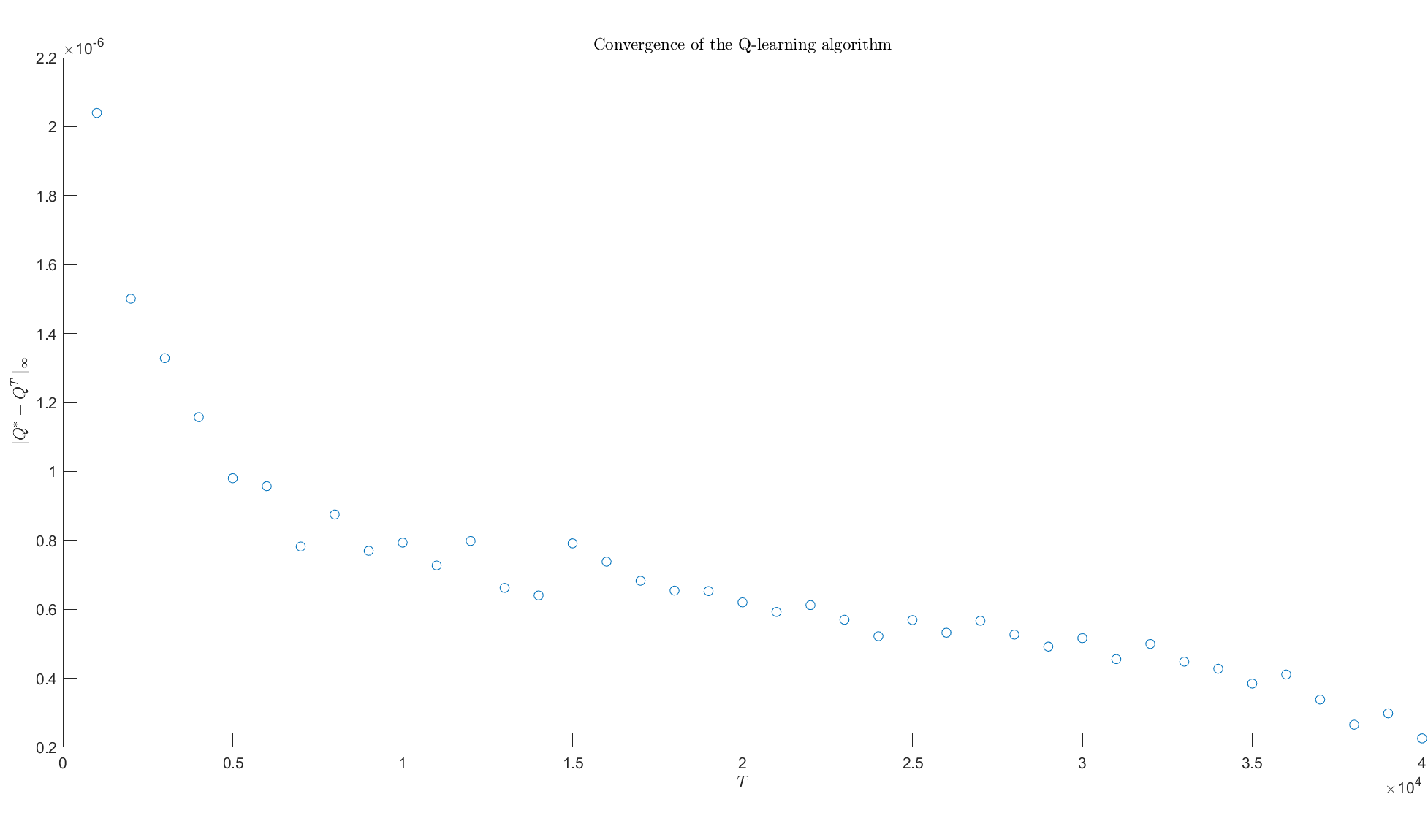}
    \captionof{figure}{Convergence of the quantized Q-learning algorithm for Example II}
    \label{fig:Q_plot_ex_II}
\end{minipage}

\begin{remark}
It is important to note that a significant challenge, when implementing Q-learning for the K-step periodic information sharing pattern, is that the dimension of the action space is very large. In fact, even in the simple case when $K=2$, we get that the dimension of the action space is $\displaystyle \prod\limits_{i=1}^{N}|\mathbb{U}^{i}|^{|\mathbb{Y}^{i}|^{2}+|\mathbb{Y}^{i}|}$. This is a computationally demanding problem even when all the parameters of the original problem are relatively simple.
\end{remark}

\subsection{Sliding memory periodic information sharing pattern} \label{sliding mem numerics}
Consider the MDP introduced in Theorem \ref{theorem8} and recall that without loss in generality the local policies of each agent can be replaced with their actions $u^{i}$. For $M=2$ $K=2$, we first estimate using empirical measures the transition kernel of the resulting kernel and then apply value iteration. Here, again $\beta=0.01$. Table \ref{tab:SW_table} represents the optimal value for select states. In particular the first eighth columns correspond to a given window of common information and the ninth column represents the corresponding optimal value.
\begin{table}[h!] 
\centering
\caption{Value iteration for sliding memory periodic information sharing pattern}
\label{tab:SW_table}
\begin{tabular}{|c|c|c|c|c|c|c|c|c|} 
\hline
$y^{1}_{0}$ & $y^{1}_{1}$ & $y^{2}_{0}$ & $y^{2}_{1}$ & $u^{1}_{0}$ & $u^{1}_{1}$ & $u^{2}_{0}$ & $u^{2}_{1}$ & $V^{*}$ \\ 
\hline
0 & 0 & 0 & 0 & 0 & 0 & 0 & 0 & 0.509596 \\ \hline
0 & 0 & 0 & 0 & 0 & 0 & 0 & 1 & 1.261110 \\ \hline
0 & 0 & 0 & 0 & 0 & 0 & 1 & 0 & 1.510460 \\ \hline
0 & 0 & 0 & 0 & 0 & 0 & 1 & 1 & 1.344227 \\ \hline
0 & 0 & 0 & 0 & 0 & 1 & 0 & 0 & 0.509596 \\ \hline
0 & 0 & 0 & 0 & 0 & 1 & 0 & 1 & 1.261110 \\ \hline
0 & 0 & 0 & 0 & 0 & 1 & 1 & 0 & 1.510460 \\ \hline
0 & 0 & 0 & 0 & 0 & 1 & 1 & 1 & 1.344227 \\ \hline
0 & 0 & 0 & 0 & 1 & 0 & 0 & 0 & 0.509596 \\ \hline
0 & 0 & 0 & 0 & 1 & 0 & 0 & 1 & 1.261110 \\ \hline
0 & 0 & 0 & 0 & 1 & 0 & 1 & 0 & 1.510460 \\ \hline
0 & 0 & 0 & 0 & 1 & 0 & 1 & 1 & 1.344227 \\ \hline
0 & 0 & 0 & 0 & 1 & 1 & 0 & 0 & 0.509596 \\ \hline
0 & 0 & 0 & 0 & 1 & 1 & 0 & 1 & 1.261110 \\ \hline
0 & 0 & 0 & 0 & 1 & 1 & 1 & 0 & 1.510460 \\ \hline
0 & 0 & 0 & 0 & 1 & 1 & 1 & 1 & 1.344227 \\ \hline
0 & 0 & 0 & 1 & 0 & 0 & 0 & 0 & 0.509596 \\ \hline
0 & 0 & 0 & 1 & 0 & 0 & 0 & 1 & 1.261110 \\ \hline
0 & 0 & 0 & 1 & 0 & 0 & 1 & 0 & 1.510460 \\ \hline
\end{tabular}
\end{table}

\section{Conclusion} 
In this paper, we proved that finite memory policies are near optimal for the KSPISP by replacing the ever increasing common information with a finite sliding window of information. We then rigorously showed that Q-learning can be employed to arrive at near optimal solutions under a variety of constructions. It is also worth noting that while we have focused on the case of finite measurement spaces for the QSWPSP, in the case of continuous measurements, since the latter become part of the state in the equivalent MDP $(\hat{Z}_{q},a_{q})$ one can discretize the measurements (thus leading to a discretization of the extended state $\hat{Z}_{q}$) in view of Proposition \ref{prop tilde theta WF} as in \cite{SaldiYukselLinder2020Asymptoticoptimalityoffinitemodelapproximations}. The QSWPSP could then be used to obtain a near optimal solution for such continuous valued measurements as well.  



\section{Appendix}

\subsection{Proof of Proposition \ref{prop tilde theta WF}}
Consider a sequence \[(\pi_{q-M}^{n},I^{M}_{q,n},a_{q,n})\rightarrow (\pi_{q-M},I^{M}_{q},a_{q})\] as $n\rightarrow \infty$ and let $f$ be a continuous and bounded function. Then, it is sufficient to show that 
\begin{eqnarray*}
    &&\int f(\Bar{G}(\pi_{q-M}^{n},I^{M}_{q,n}),\mathbf{y}_{[(q+1-M)K,qK-1]}^{n},\\
    &&a_{[(q+1-M),q-1]}^{n},a_{q}^{n},\mathbf{y}_{[qK,qK+K-1]})\\
    &&P\bigg(dy_{[qK,qK+K-1]}^{[1,N]} \mid x_{qK}, a_{q}^{n}\bigg)\phi(\pi_{q-M}^{n},I^{M}_{q,n})(dx_{qK})\\
    && \rightarrow \int f(\Bar{G}(\pi_{q-M},I^{M}_{q}),\mathbf{y}_{[(q+1-M)K,qK-1]},\\
    &&a_{[(q+1-M),q-1]},a_{q},\mathbf{y}_{[qK,qK+K-1]})\\
    &&P\bigg(dy_{[qK,qK+K-1]}^{[1,N]} \mid x_{qK}, a_{q}\bigg)\phi(\pi_{q-M},I^{M}_{q})(dx_{qK})\\
\end{eqnarray*}

By the generalized DCT theorem \cite[Theorem 3.5]{serfozo1982convergence}, and the continuity of $\Bar{G}$, which itself follows from the continuity of $G$ (see the proof of \cite[Theorem 4.1]{OMZSY-centralized-reduction}), it is sufficient to check that 
\begin{eqnarray*}
&&P\bigg(dy_{[qK,qK+K-1]}^{[1,N]} \mid x_{qK}, a_{q}^{n}\bigg)\phi(\pi_{q-M}^{n},I^{M}_{q,n})(dx_{qK})\rightarrow \\ 
&& \quad \quad P\bigg(dy_{[qK,qK+K-1]}^{[1,N]} \mid x_{qK}, a_{q}\bigg)\phi(\pi_{q-M},I^{M}_{q})(dx_{qK})    
\end{eqnarray*} weakly. Again, by \cite[Theorem 4.3]{OMZSY-centralized-reduction}, we have that  $\phi(\pi_{q-M}^{n},I^{M}_{q,n})(dx_{qK})\rightarrow \phi(\pi_{q-M},I^{M}_{q})(dx_{qK})$ weakly. Thus, again by the generalized DCT, it is sufficient to check that 

\begin{eqnarray*}
   && P\bigg(dy_{[qK,qK+K-1]}^{[1,N]} \mid x_{qK}^{n}, a_{q}^{n}\bigg)\rightarrow \\
   &&P\bigg(dy_{[qK,qK+K-1]}^{[1,N]} \mid x_{qK}, a_{q}\bigg)
\end{eqnarray*} weakly for any $x_{qK}^{n} \rightarrow x_{qK}$. Let $\rho$ denote the bounded Lipschitz metric as defined in Section \ref{convergence of probability measures}. Next, we have that 
\begin{eqnarray*}
  && \rho(P\bigg(dy_{[qK,qK+K-1]}^{[1,N]} \mid x_{qK}^{n}, a_{q}^{n}\bigg), \\
&& \qquad \qquad P\bigg(dy_{[qK,qK+K-1]}^{[1,N]} \mid x_{qK}, a_{q}\bigg))\\
  &&\leq \rho(P\bigg(dy_{[qK,qK+K-1]}^{[1,N]} \mid x_{qK}^{n}, a_{q}^{n}\bigg),\\
  && \qquad \qquad P\bigg(dy_{[qK,qK+K-1]}^{[1,N]} \mid x_{qK}^{n}, a_{q}\bigg))\\
  && \quad+\rho(P\bigg(dy_{[qK,qK+K-1]}^{[1,N]} \mid x_{qK}^{n}, a_{q}\bigg),\\
  && \qquad \qquad P\bigg(dy_{[qK,qK+K-1]}^{[1,N]} \mid x_{qK}, a_{q}\bigg)) \\
  && \leq \|P\bigg(dy_{[qK,qK+K-1]}^{[1,N]} \mid x_{qK}^{n}, a_{q}^{n}\bigg)\\
  && \qquad \qquad -P\bigg(dy_{[qK,qK+K-1]}^{[1,N]} \mid x_{qK}^{n}, a_{q}\bigg)\|_{TV}\\
  && \quad+\|P\bigg(dy_{[qK,qK+K-1]}^{[1,N]} \mid x_{qK}^{n}, a_{q}\bigg)\\
  && \qquad \qquad -P\bigg(dy_{[qK,qK+K-1]}^{[1,N]} \mid x_{qK}, a_{q}\bigg)\|_{TV} 
\end{eqnarray*} Now, by the proof 
of \cite[Theorem 4.3]{OMZSY-centralized-reduction} it follows that
\begin{eqnarray*}
    &&\|P\bigg(dy_{[qK,qK+K-1]}^{[1,N]} \mid x_{qK}^{n}, a_{q}^{n}\bigg)\\
    && \quad -P\bigg(dy_{[qK,qK+K-1]}^{[1,N]} \mid x_{qK}^{n}, a_{q}\bigg)\|_{TV}\rightarrow 0
\end{eqnarray*}
 Moreover, by the continuity of $\tau$ in total variation, we have that 
\begin{eqnarray*}
 && \|P\bigg(dy_{[qK,qK+K-1]}^{[1,N]} \mid x_{qK}^{n}, a_{q}\bigg)\\
&& \quad -P\bigg(dy_{[qK,qK+K-1]}^{[1,N]} \mid x_{qK}, a_{q}\bigg)\|_{TV}\rightarrow 0   
\end{eqnarray*}
 \qed

\vspace{-1.5cm}
\begin{IEEEbiography}
{\bf Omar Mrani-Zentar} received his B.Sc. and M.Sc. degrees in Mathematics from the University of British Columbia and is currently a PhD student in Mathematics and Statistics at Queen's University. His research interests include stochastic control theory and probability.
\end{IEEEbiography}

\vspace{-2.55cm}

\begin{IEEEbiography}
{\bf Serdar Y\"uksel} (S'02, M'11) received his B.Sc. degree in Electrical and Electronics Engineering from Bilkent University in 2001; M.S. and Ph.D. degrees in Electrical and Computer Engineering from the University of Illinois at Urbana-
Champaign in 2003 and 2006, respectively. He was a post-doctoral researcher at Yale University before joining Queen's University as an Assistant Professor in the Department of Mathematics and Statistics, where he is now a Professor. His research interests are on stochastic control, decentralized control, information theory and probability. He is a co-author of several books and has been an editor with several journals. 
\end{IEEEbiography}

\end{document}